\documentclass[10pt]{amsart}
\usepackage{bbm}
\usepackage[normalem]{ulem}
\usepackage{mathrsfs}
\usepackage[usenames,dvipsnames,svgnames,table]{xcolor}
\usepackage[pdftex,pagebackref,colorlinks=true,linkcolor=BrickRed,citecolor=OliveGreen,pdfstartview=FitH]{hyperref}

\usepackage{bbm}
\usepackage[normalem]{ulem}
\usepackage{mathrsfs}
\usepackage[usenames,dvipsnames,svgnames,table]{xcolor}
\usepackage[pdftex,pagebackref,colorlinks=true,linkcolor=BrickRed,citecolor=OliveGreen,pdfstartview=FitH]{hyperref}

\usepackage{graphicx}
\usepackage{hyperref}
\usepackage{float}
\usepackage{hyperref}
\usepackage{enumitem}
\usepackage{verbatim}

\usepackage{tikz}
\usepackage{tikz-cd}
\usetikzlibrary{calc}
\def\radius{0.05}
\def\tikzscale{1.5} 
\newcommand{\tile}[2]{
  \draw (#1, #2) rectangle (#1 + 1, #2 + 1);
  \filldraw (#1, #2) circle (\radius);
  \filldraw (#1 + 1, #2) circle (\radius);
  \filldraw (#1, #2 + 1) circle (\radius);
  \filldraw (#1 + 1, #2 + 1) circle (\radius);
}

\definecolor{plum}{rgb}{0.56, 0.27, 0.52}
\definecolor{violet}{rgb}{0.58, 0.0, 0.83}
\definecolor{cornflowerblue}{rgb}{0.39, 0.58, 0.93}
\definecolor{slate}{rgb}{0.28, 0.24, 0.55}

\def\centerarc[#1](#2)(#3:#4:#5)
    { \draw[#1] ($(#2)+({#5*cos(#3)},{#5*sin(#3)})$) arc (#3:#4:#5); }

\newtheorem{thm}{Theorem}[section]

\newtheorem{lem}[thm]{Lemma}
\newtheorem{prop}[thm]{Proposition}

\theoremstyle{definition}

\theoremstyle{remark}

\newtheorem{remk}[thm]{Remark}

\theoremstyle{definition}

\newcommand{\Rbb}{ {\mathbb R}}
\newcommand{\Zbb}{ {\mathbb Z}}

\newcommand{\Cbb}{ {\mathbb C}}
\newcommand{\Qbb}{ {\mathbb Q}}

\newcommand{\oY}{ \overline{Y}}
\newcommand{\oX}{ \overline{X}}
\newcommand{\op}{ \overline{p}}
\newcommand{\oq}{ \overline{q}}

\newcommand{\oL}{ \overline{L}}
\newcommand{\oLambda}{ \overline{\Lambda}}

\newcommand{\tgamma}{ \widetilde{\gamma}}
\newcommand{\tp}{ \widetilde{p}}
\newcommand{\tP}{ \widetilde{P}}
\newcommand{\tK}{ \widetilde{K}}

\newcommand{\Ecal}{ {\mathcal E}}

\newcommand{\Mcal}{ {\mathcal M}}

\newcommand{\Tcal}{ {\mathcal T}}

\newcommand{\ind}{\rm{ind}}

\renewcommand{\Im}{\mathrm{Im}}

\def\<{\left\langle}
\def\>{\right\rangle}

\newcommand{\cjin}[1]{\textcolor{purple}{#1}}
\newcommand{\cjout}[1]{\textcolor{purple}{\sout{#1}}}

\title{Haupt's theorem for strata of abelian differentials}

\author{Matt Bainbridge}

\address{Department of Mathematics, Indiana University, 
Bloomington, IN, 47401, USA }
\email{bainbridge.matt@gmail.com}

\author{Chris Johnson}
\address{Department of Mathematics and Computer Science, Western Carolina University, 
Cullowhee, NC 28723, USA}
\email{ccjohnson@gmail.com}

\author{Chris Judge}
\address{Department of Mathematics, Indiana University, 
Bloomington, IN, 47401, USA }
\email{cjudge@indiana.edu}

\author{Insung Park}
\address{Department of Mathematics, Indiana University, 
Bloomington, IN, 47401, USA }
\email{park433@iu.edu}

\date{\today}

\thanks{The work of C.J.\  is partially supported by a Simons collaboration grant. }
\thanks{The work of M.B.\  is partially supported by a Simons collaboration grant.}

\begin{document}

\begin{abstract}
Let $S$ be a closed topological surface. Haupt's theorem provides necessary and sufficient conditions for a complex-valued character of the first integer
homology group of $S$ to be realized by integration against
a complex-valued 1-form  that is holomorphic with respect to some
complex structure on $S$. We prove a refinement of this theorem 
that takes into account the divisor data of the 1-form. 
\end{abstract}

\maketitle

\section{Introduction}

Let $S$ be an oriented connected topological 
surface without boundary having genus $g \geq 2$. 
We say that a character 
$\chi\colon H_1(S; \Zbb) \to \Cbb$ is {\em realized}
by a complex-valued 1-form $\omega$ if and only if for each integral 
cycle $\gamma$ we have $\int_{\gamma} \omega = \chi(\gamma)$. 
In this case, the image $\Lambda_{\chi}$ of $\chi$ is the set of {\em periods}
of $\omega$.

In 1920, O.\ Haupt \cite{Haupt} determined those characters 
that are realized by some
1-form that is holomorphic with respect to some 
complex structure on $S$. More recently, M. Kapovich \cite{Kapovich} 
rediscovered Haupt's characterization in the following form: 
A character $\chi$ is realized by a holomorphic
1-form $\omega$ if and only if 

\begin{enumerate}
\item its {\em area} $A(\chi):= \Im\sum \overline{\chi(a_i)}  \chi(b_i)$ 
is positive 
where $\{a_i,b_i\}$ is a symplectic basis of $H_1(S; \mathbb{Z})$, and 

\item  if $\Lambda_{\chi}$ is discrete,
then $\Lambda_{\chi}$ is a lattice and the induced homotopy class of maps 
from $S$ to the torus $\Cbb/\Lambda_{\chi}$  has degree 
$d_{\chi}$ strictly greater than 1.
\end{enumerate}
In addition, if $\Lambda_{\chi}$ is discrete, then the induced map 
is realized by a branched covering $p\colon S \to \Cbb/\Lambda_{\chi}$ 
and the pullback $p^*(dz)$ realizes $\chi$. 

In this note we provide a refinement of Haupt's theorem that involves
the {\em divisor data} of the 1-form. To be precise, let 
$Z(\omega)= \{z_1, z_2, \ldots, z_k\}$ be the set of zeros of 
a nontrivial holomorphic 
1-form $\omega$, and for each $i$ let $\alpha_i$ denote the multiplicity 
of the zero $z_i$.
The divisor data, $\alpha(\omega)$, is the unordered $n$-tuple  
$(\alpha_1, \ldots, \alpha_k)$, whose sum is 
$\alpha_1 + \alpha_2 + \cdots + \alpha_k = 2g-2$.

\begin{thm} \label{thm:main}
A character $\chi: H_1(S, \Zbb) \to \Cbb$ is realized by a 1-form 
$\omega$ with divisor data $\alpha(\omega)= (\alpha_1, \ldots, \alpha_k)$
if and only if
\begin{enumerate}
\item[(1)] $A(\chi)$ is positive, and

\item[(2')]   if $\Lambda_{\chi}$ is discrete, then the induced map
$S \to \Cbb/ \Lambda_{\chi}$ has degree $d_{\chi} >\max{\{\alpha_i\}}$. 
\end{enumerate}
\end{thm}

The proof of the sufficiency is immediate. Indeed, one
applies Haupt's theorem and notes that
the Riemann-Hurwitz formula 
shows that the degree of an induced branched covering 
is at least $1+\max{\{\alpha_i\}}$.

To prove the necessity, we will recast the problem in terms
of the moduli space theory of 1-forms  (see \S  \ref{sec:leaves}).
The Hodge bundle $\Omega \Mcal_g$ is the moduli space  
of complex-valued 1-forms that are holomorphic with respect 
to some complex structure on $S$.  It is a disjoint union of the strata $\Omega \Mcal_g(\alpha)$, consisting of forms with divisor data $\alpha$.
A connected component of the set of 1-forms that have a prescribed set of periods 
constitutes a leaf of the `isoperiodic foliation'. 
Calsamiglia, Deroin, and Francaviglia \cite{CDF}
classified the closures of the leaves of the 
isoperiodic foliation. We use this classification
to prove the following.

\begin{thm} 
\label{thm:intersect-each-stratum}
If $L$ is an isoperiodic leaf whose associated set of periods is
not a lattice, then $L$ intersects each connected component of 
each stratum of the Hodge bundle.
\end{thm}

To prove Theorem \ref{thm:main}, one combines Theorem 
\ref{thm:intersect-each-stratum} with the following proposition.

\begin{prop} 
\label{prop:covering}
For any connected component $K$ of a stratum $\Omega \Mcal_g(\beta)$ of $\Omega \Mcal_g$
and for each integer $d > \max{\{\beta_k\}}$, 
there exists a primitive degree $d$ branched covering 
$p\colon  S \to \Cbb/ (\Zbb+ i \Zbb)$ such that $(S,p^*(dz))$ belongs to $K$.
\end{prop}

Recall that 
a branched cover of a torus is primitive if the induced map on homology is surjective.

In \S \ref{sec:leaves}, we construct the Hodge bundle over Teichm\"uller space, 
define the isoperiodic foliation, recall the main result of 
\cite{CDF}, and prove Theorem \ref{thm:intersect-each-stratum}.
In \S \ref{section:lattice}, we prove 
Proposition \ref{prop:covering}.

Soon after we posted this paper on the arxiv,  Thomas Le Fils 
shared a preprint \cite{LFs20} containing his independent proof of Theorem 
\ref{thm:main}. His proof differs from ours in that it does not pass 
through Theorem \ref{thm:intersect-each-stratum} and instead uses 
a study of the mapping class group action on the space of characters
in the spirit of \cite{Kapovich}. We note that his paper 
does not consider 
connected components of strata.

We thank an anonymous referee for very helpful comments.


\section{The Hodge bundle and the isoperiodic foliation} \label{sec:leaves}

In this section we describe the Hodge 
bundle and the absolute and relative period mappings.
We define the isoperiodic foliation and show that 
each leaf that passes near a stratum must intersect the 
stratum. We use this to prove Theorem \ref{thm:intersect-each-stratum}.
Finally we prove Theorem \ref{thm:main} modulo the 
proof of Proposition \ref{prop:covering}.

We begin by describing the Hodge bundle as a bundle over 
Teichm\"uller space.
A \emph{marked Riemann surface} is a closed Riemann surface $X$ 
together with an orientation-preserving homeomorphism $f\colon S \to X$.  
Two marked surfaces $(f_1, X_1)$ and $(f_2, X_2)$ are considered to 
be equivalent if $f_2 \circ f_1^{-1}$ is isotopic to a conformal map.  
The set of equivalence classes of marked genus $g$ surfaces may be given 
the structure of a complex manifold homeomorphic to  
$\mathbb{C}^{3g-3}$ called the Teichm\"uller space $\Tcal_g$.  

The \emph{Hodge bundle} $\Omega \Tcal_g\to\Tcal_g$ is the (trivial)
vector bundle over $\Tcal_g$  whose fiber above $(f,X)$ 
consists of (equivalence classes of) holomorphic 
$1$-forms on $X$. In other words, $\Omega \Tcal_g$ 
is the space of triples $(f, X, \omega)$ up to natural equivalence. 
The total space of $\Omega \Tcal_g$ is naturally a complex manifold
of dimension $4g-3$.
The {\em absolute period map} $P\colon \Omega\Tcal_g\to H^1(S; \mathbb{C})$
is the holomorphic map that assigns to each triple
$(f, X, \omega)$ the cohomology class $f^*(\omega)$.  

Let $\Omega^*\Tcal_g\subset \Omega \Tcal_g$ denote the set of one-forms
that do not vanish identically.  The map that assigns divisor data to each
$1$-form defines a stratification of $\Omega^* \Tcal_g$.  In
particular, for each partition $\alpha = (\alpha_1, \ldots, \alpha_k)$ of
$2g-2$, we define the {\em stratum} $\Omega\Tcal_g(\alpha)$ to consist of 
those triples $(f, X, \omega)$ such that the
divisor data of $\omega$ equals $\alpha$.

One may also define a relative period map in a neighborhood
of each non-trivial marked one-form $(f_0, X_0, \omega_0)$ in the stratum 
$\Omega \Tcal_g(\alpha)$.  Let $Z\subset S$ be a set of $k$ marked points.  
Over a contractible  neighborhood $U \subset \Omega \Tcal_g(\alpha)$
of $(f_0, X_0, \omega_0)$, one may choose representative marking maps to identify $Z$ with the zero sets
$Z(\omega)$.  Pulling back by these marking maps the class $[\omega]\ \in H^1(X, Z(\omega);
\mathbb{C})$ then defines the \emph{relative period map} 
$P_\mathrm{rel}\colon U\to H^1(S, Z; \mathbb{C})$.

The relative period map is well-known to be a local biholomorphism \cite{Veech}.
Moreover, the relative and absolute period
maps are related by $P|_{U} = r \circ P_{\mathrm{rel}}$   
where $r$ is the natural map from $H^1(S, Z; \Cbb)$ to $H^1(S; \Cbb)$.
By considering the long exact sequence in cohomology, one finds that 
$r$ is surjective, and hence $P|_U$ is a submersion. Since 
every non-trivial one-form lies in some stratum, we have the following.

 \begin{lem}
    \label{lem:period_submersion}
     The restriction of the absolute period map $P$ to $\Omega^* \Tcal_g$
     is a submersion, as is its restriction to any stratum 
     in $ \Omega\Tcal_g$.
 \end{lem}
 
 Since $P$ is a submersion, it defines a holomorphic 
 foliation of $\Omega^*\Tcal_g$ called the 
 {\em isoperiodic (or Rel) foliation}. 
 Each {\em isoperiodic leaf} is a connected
 component of a level set of $P$. 
 
The mapping class group ${\rm Mod}(S)$ naturally 
acts biholomorphically and properly discontinuously on the Hodge bundle. 
The quotient of this action is the classical Hodge bundle
$\Omega \Mcal_g \to \Mcal_g$ where the base $\Mcal_g$ is
the moduli space of Riemann surfaces.
In particular, each point in $\Omega \Mcal_g$ may be regarded as (the
equivalence class of) a pair $(X, \omega)$ where $X$ is 
a Riemann surface and $\omega$ is a holomorphic 1-form on $X$.

If $\varphi \in {\rm Mod}(S)$ then we have 
$P(\varphi^*(\omega))=\varphi^*(P(\omega))$. 
It follows that the isoperiodic foliation descends to a 
foliation of $\Omega \Mcal_g$ that
we will also refer to as the isoperiodic foliation.
Moreover, we have a well-defined map 
from the set of leaves  to the orbit space $H^1(S; \Cbb)/ {\rm Mod}(S)$,
and the set of periods 
$\Lambda_L:=\left\{\int_{\gamma} \omega\, :\, \gamma \in H_1(S, \Zbb) \right\}$
depends only on the isoperiodic leaf $L$ to which $\omega$ belongs.

Each stratum $\Omega \Tcal_g(\alpha)$ is invariant under the action
of ${\rm Mod}(S)$. Each quotient, 
$\Omega \Mcal_g(\alpha):= \Omega \Tcal_g(\alpha)/{\rm Mod}(S)$,
is the {\em stratum} that 
consists of pairs $(X, \omega)$ with divisor data $\alpha$.

\begin{prop}  \label{prop:intersect}
    Let $K$ be a connected component of a stratum. 
    There exists a neighborhood $Z \subset \Omega \Mcal_g$ 
    of $K$ such that if an isoperiodic leaf $L$ intersects
    $Z$, then $L$ also intersects $K$. 
\end{prop}

\begin{proof}
Let $\tK$ be a connected component of the preimage of $K$ in $\Omega^* \Tcal_g$.  
By Lemma~\ref{lem:period_submersion}, 
the map $P$ is a holomorphic submersion from the $4g-3$ dimensional 
complex manifold $\Omega^* \Tcal_g$ onto the complex vector space
$H^1(S;\Cbb)$ which has dimension $2g$. Thus, given 
$(f,X, \omega)$, the inverse function 
theorem  provides an open ball $B^{2g-3} \subset \Cbb^{2g-3}$,
an open ball $B^{2g} \subset H^1(S; \Cbb)$, and a biholomorphism 
$\varphi$ from $B^{2g-3} \times B^{2g}$ onto a neighborhood $U$
of $(f,X, \omega)$ so that  $P \circ \varphi(z,w)= w$. 

Suppose that $(f,X, \omega)$ lies in $\tK$. 
Since the restriction of $P$ to $\tK$ is a submersion,
the image $V:=P(U \cap \tK)$ is open.
Note that $(P \circ \varphi)^{-1}(V)= B^{2g-3} \times V$. 
If $L$ is a connected component of $P^{-1}(\chi)$ that 
intersects $W:=\varphi(B^{2g-3} \times V)$, then 
$\chi \in V$ and $L \cap U = \varphi(B^{2g-3} \times \{\chi\})$. 
In particular, $L$ intersects $\tK$. 

The neighborhood $Z$ is constructed by taking the image in $\Omega \Mcal_g$ of the union 
of all such neighborhoods $W$ as $(f, X, \omega)$ varies over 
$\tK$. 
\end{proof}

Next, we describe the result of Casamiglia, Deroin, and 
Francaviglia \cite{CDF} that classifies the closures of leaves $L$
in terms of the associated set of periods $\Lambda_L$. 
    The  closure, $\overline{\Lambda}_{L}$ is a closed real Lie subgroup of 
$\Cbb \cong \Rbb^2$. Thus,  $\overline{\Lambda}_{L}$
is either  equal to $\Cbb$, is isomorphic to  $\Zbb \oplus \Rbb$, or is discrete.

Let $\Omega_1 \Mcal_g\subset \Omega\Mcal_g$ denote the locus of unit-area forms.  
Since the area functional 
$A(\omega)= \frac{i}{2}\int_S \omega \wedge \overline{\omega}$ depends only on absolute
periods, $\Omega_1 \Mcal_g$ is saturated by leaves of the isoperiodic foliation.

Given any closed subgroup $\Gamma\subset \mathbb{C}$, let
$\Omega_1^\Gamma \Mcal_g\subset \Omega_1\Mcal_g$ denote the
union of the leaves $L$ such that there exists 
a connected subgroup $\Gamma' \subset \Gamma$ with 
$\Gamma = \overline{\Lambda}_L + \Gamma'$.
If $\Gamma = \mathbb{C}$, then $\Omega_1^\Gamma\Mcal_g = \Omega_1 \Mcal_g$.  
If $\Gamma$ is isomorphic to $\Rbb + \sqrt{-1} \cdot \Zbb$,
then $L \subset \Omega_1^\Gamma\Mcal_g$ if either 
$\overline{\Lambda}_L = \Gamma$ 
or $\overline{\Lambda}_L$ 
is a discrete subgroup of $\Gamma$ with `primitive imaginary part'. 
If $\Gamma$ is discrete, then $\Omega_1^\Gamma\Mcal_g$ is nonempty only 
if $\Gamma$ has covolume $1/d$ for some integer $d>1$,
in which case $\Omega_1^\Gamma\Mcal_g$ is a closed isoperiodic leaf which 
parameterizes primitive degree $d$ branched covers of $\mathbb{C}/\Gamma$.

\begin{prop}
If $\Gamma$ is a lattice, then
the space $\Omega_1^\Gamma\Mcal_g$ is connected. 
\end{prop}

\begin{proof}
By Theorem 9.2 of \cite{GabKaz},
given two primitive, simply branched coverings $p:S \to \Cbb/\Gamma$ 
and $q:S \to \Cbb/\Gamma$ of the same degree, there exists a homeomorphism 
$h: S \to S$ and a homeomorphism $k: \Cbb/\Gamma \to \Cbb/\Gamma$ 
isotopic to the identity so that $k \circ p = q \circ h$. 
Let $k_t$ be the isotopy with $k_0 = k$ and $k_1={\rm id}$.
For each $t$, the 1-form $(k_t \circ p)^*(dz)$ is holomorphic with 
respect to the pulled-back complex structure.
We have $(k_0 \circ p)^*(dz)= h^*\left( q^*(dz) \right)$
and $(k_1 \circ p)^*(dz)= p^*(dz)$. Hence the path 
in $\Omega_1^{\Gamma} \Mcal_g$ associated to $(k_t \circ p)^*(dz)$
joins the point represented by 
$q^*(dz)$ to the point represented by $p^*(dz)$.
Since simply branched coverings are generic, 
the space $\Omega_1^{\Gamma} \Mcal_g$ is connected.
\end{proof}

Because $\Omega_1^{\Gamma} \Mcal_g$ is connected, we may simplify 
the statement of the main theorem of \cite{CDF}.

\begin{thm}[\cite{CDF}] \label{thm:CDF}
Let $L\subset \Omega_1 \Mcal_g$ be a leaf of the 
isoperiodic foliation and let $\Gamma = \overline{\Lambda}_L$.  
If $g>2$, then the closure of $L$ is $\Omega_1^\Gamma\Mcal_g$. 
If $g=2$, then either the closure
of $L$ is $\Omega_1^\Gamma\Mcal_g$ or $L$ lies in the eigenform
locus $\Ecal \subset \Omega_1 \Mcal_g$. 
\end{thm}

We are now ready to prove Theorem \ref{thm:intersect-each-stratum}. 

\begin{proof}[Proof of Theorem \ref{thm:intersect-each-stratum}]
  We first suppose that $g >2$ or $g=2$ and $L \not\subset \Ecal$.
  If $L$ is an isoperiodic leaf such that $\Lambda_L$
  is not a lattice, then $\oLambda_L$ either equals $\Cbb$ or 
  equals $\Rbb \cdot z_1 \oplus \Zbb \cdot z_2$ where $z_i \in \Cbb$.
  By Lemma \ref{lem:period_submersion} the restriction of the 
  absolute period map to a given component $K$ of a given stratum 
  is an open map. It follows that 
  there exists $(X, \omega) \in K$ of area 1 so that the periods of 
  $\omega$ lie in $\Qbb \cdot z_1 \oplus \Qbb \cdot z_2$. In particular, 
  the set of periods constitute a lattice and there exists 
  $A \in SL_2(\Rbb)$ so that the periods of $A \cdot (X, \omega)$
  lie in $\oLambda_L$.  Hence $A \cdot (X,\omega)$ lies in the closure
  $\oL$ by Theorem \ref{thm:CDF}. Thus $K$ intersects $\oL$, and hence $K$ 
  intersects $L$ by Proposition~\ref{prop:intersect}.
  
  It remains to consider the case where $g=2$ and $L \subset \Ecal$.
  In this case, Theorem \ref{thm:intersect-each-stratum} follows from work of McMullen
  \cite{McM03, McM05}. Indeed, $\Omega \Mcal_2$ consists of two strata, the principal stratum $\Omega \Mcal_2(1,1)$ and the stratum $\Omega \Mcal_2(2)$,
  and both of these strata are connected. 
  McMullen shows that the eigenform locus 
  $\mathcal{E}\subset \Omega_1 \Mcal_2$ is a countable union of orbifolds $\Omega_1 E_D$
  where $D$ belongs to a subset of the positive integers.  Moreover, 
  each $\Omega_1 E_D$ is saturated by leaves of the isoperiodic foliation.  
  The intersection  $\Omega_1 E_D \cap \Omega_1 \Mcal_2(2)$ is his ``Weierstrass curve''
  $\Omega_1 W_D$.   The eigenform locus $\Omega_1 E_D$ is a circle bundle over a Hilbert modular surface, which is covered by $\mathbb{H} \times \mathbb{H}$.  In this covering, the isoperiodic foliation is simply the ``vertical'' foliation with leaves $\{c\}\times \mathbb{H}$.   Each component of the Weierstrass curve is covered by a graph of a holomorphic function $\mathbb{H}\to\mathbb{H}$ which \emph{a fortiori} must intersect each vertical leaf, and hence  every isoperiodic leaf in $\Omega_1 E_D$ must intersect $\Omega_1 W_D$. 
    Finally, each $\Omega_1 W_D$ is nonempty unless $D=4$, 
  in which case $\Omega_1 E_4$ parameterizes degree $2$ torus-covers, 
  a case that is excluded by Theorem~\ref{thm:main}. 
\end{proof}

We remark that if $\Lambda_L$ is a lattice, then the 
associated space $\Omega_1^{\Lambda_L} \Mcal_2$ need not intersect every stratum
$\Omega \Mcal_2(\alpha)$.  Indeed, for such an intersection to be nonempty, 
it is necessary for the covolume of $\Lambda_L$ to be strictly 
less than $1/ \max{\alpha_i}$. Proposition \ref{prop:covering} below 
implies that this condition is also sufficient.

Finally, we prove our variant of Haupt's theorem modulo Proposition \ref{prop:covering}. 

\begin{proof}[Proof of Theorem \ref{thm:main}]
  Suppose that $\chi \in {\rm Hom}(H^1(S;\Zbb), \Cbb) \cong H^1(S;\Cbb)$ is a character which
  satisfies the hypotheses of Theorem~\ref{thm:main}.  By applying a real rescaling, we may assume
  moreover that $A(\chi)=1$.  Haupt's theorem then provides a unit-area holomorphic 1-form
  $(X,\omega)\in \Omega_1\Tcal_g$ representing $\chi$.  Let $L\subset \Omega_1 \Tcal_g$ be the
  isoperiodic leaf passing through $(X, \omega)$, with $\pi(L)$ its image in $\Omega_1 \Mcal_g$.

  If $\Lambda_L$ is not a lattice, then Theorem \ref{thm:intersect-each-stratum} implies that the
  leaf $\pi(L)$ intersects $\Omega \Mcal(\beta)$ for every $\beta$.  A form $(X', \omega')$ in this intersection is
  then a representative of $\chi$ in the desired stratum.

  If $\Lambda_L$ is a lattice, then  the image $\pi(L)$ of $L$ in
  $\Omega_1^{\Lambda_L}\Mcal_g$ is the space of primitive degree $d$ branched covers of
  $\mathbb{C}/\Lambda_L$.  Applying the $GL_2^+(\Rbb)$-action, we may assume that $\Lambda_L = \Zbb \oplus i \Zbb$. 
  Proposition 
 \ref{prop:covering} then implies that the leaf $\pi(L)$ intersects $\Omega \Mcal(\beta)$. 
\end{proof}


\section{Primitive torus covers}  \label{section:lattice}

In this section we complete the proof of Theorem \ref{thm:main} 
by proving Proposition~\ref{prop:covering}.  That is, for each connected component 
$K$ of a stratum $\Omega \Mcal(\alpha)$, we construct a primitive branched torus covering $p \colon S \to \Cbb/(\Zbb + i\Zbb)$ 
so that $p^*(dz)$ lies in $K$.

To prove Theorem \ref{thm:main}, we will explicitly construct torus coverings that 
lie in connected components of strata having one or two zeros, and 
then we apply a sequence of `surgeries' to obtain torus coverings with additional zeros. 
In \S \ref{subsec:minimal-strata} we construct torus
coverings for each connected component of each minimal stratum $\Omega \Mcal(2g-2)$.
In \S  \ref{subsec:two-zeros} we construct covers 
for each component of $\Omega \Mcal_g(g-1,g-1)$. 
In \S \ref{subsec:other} we introduce surgeries that add zeros to 
a torus cover while preserving the degree, and we check the effect 
of surgery on the spin parity. In \S \ref{subsec:highest-order-odd}
we construct torus covers such that the 1-form has exactly two zeros
and each zero has odd order. 
we use surgeries to construct torus covers when $\max \alpha_i$ is odd. 
In \S \ref{sec:algorithm} we describe the algorithm that can be used to 
construct a torus cover any desired connected component. 
We also provide some examples.

In what follows we will let $T$ denote the `unit square' torus
$\Cbb/(\Zbb + i \Zbb)$.

According to \cite{KoZo}, the connected components of 
strata are distinguished by hyperellipticity
and spin parity. To be precise, we
will need to determine whether a torus covering $p \colon S \to T$
admits a {\em hyperelliptic involution}, a holomorphic 
involution $\tau\colon S \to S$ such that  the quotient $S/\langle \tau \rangle$ 
is a sphere. 
Because $\tau^*(\omega)=-\omega$, a hyperelliptic involution maps each vertical 
(resp. horizontal) cylinder to a vertical (resp. horizontal) cylinder. 
Moreover, if $\tau$ preserves a vertical or horizontal cylinder $C$,
then $\tau$ preserves the central curve of the cylinder
and fixes exactly two points on the central curve. 
The Riemann-Hurwitz formula implies that $\tau$ has 
exactly $2g+2$ fixed points. 

We will also need to check the {\em spin parity} of a holomomorphic
1-form.
Given a Riemann surface $X$ with a holomorphic one-form $\omega$ and a 
loop $\gamma\colon S^1 \to X$ disjoint from the zeros of $\omega$, the Gauss map $G_\gamma\colon S^1 \to S^1$ is defined by 
\begin{equation*}
    G_{\gamma}(t)~ =~ \frac{\omega(\gamma'(t))}{|\omega(\gamma'(t))|}.
\end{equation*}
The {\em index of $\gamma$} is the degree of $G_\gamma$.
Note that if $\gamma$ is a geodesic with respect to the natural 
the flat structure on the surface, 
then $G_{\gamma}$ is a constant map and hence $\ind(\gamma)=0$.

Following Thurston and Johnson \cite{Johnson}, Kontsevich and Zorich \cite{KoZo} 
gave the following formula for the spin parity of a holomorphic 1-form $\omega$
all of whose zeros have even order. 
Given a symplectic basis  ${a_1, b_1, ..., a_g, b_g}$ 
for $H_1(X; \Zbb)$ consisting of curves that do not pass through a zero,
the \emph{spin parity} of $\omega$ equals
\begin{equation}
\label{eqn:spin-parity}
\sum_{i=1}^g~ (\mathrm{ind}(a_i) + 1) (\mathrm{ind}(b_i) + 1) \pmod{2}.
\end{equation}
In particular, this invariant of a holomorphic 1-form with zeros of even order
lies in $\Zbb/2 \Zbb$. We refer to a 1-form as {\em even} 
if its spin parity equals 0 mod 2, and as {\em odd} otherwise.


\subsection{Minimal strata}  
\label{subsec:minimal-strata}

In this subsection, for each $d>2g-2$, we construct a degree $d$ 
primitive branched torus covering for each connected component 
of  the `minimal stratum' $\Omega \mathcal{M}_g(2g-2)$. 
For $g \geq 4$, the minimal stratum 
has exactly three connected components \cite{KoZo}: 

\begin{itemize}

\item {\em hyperelliptic:} 
      The 1-forms in $\Omega\Mcal_g(2g-2)$
      that are canonical double covers of 
      meromorphic quadratic differentials
      on the Riemann sphere with one zero of order $2g-2$ and $2g+1$
      simple poles.

\item {\em even:}  The non-hyperelliptic 1-forms with even spin parity.
     
\item {\em odd:}  The non-hyperelliptic 1-forms with odd spin parity.

\end{itemize}
Denote these components by $\Omega\mathcal{M}_g(2g-2)^{\text{hyp}}$,  $\Omega\mathcal{M}_g(2g-2)^{\text{odd}}$, and $\Omega \mathcal{M}_g(2g-2)^{\text{even}}$. 
In the case $g=3$, there is no even component,
and in the case $g=2$, there is only the hyperelliptic component \cite{KoZo}.

For each of the above connected components we will first construct 
a degree $2g-1$ primitive branched cover $p$ so that 
$p^*(dz)$ lies in the component. A slight modification
of the construction will provide primitive branched coverings of each degree $d>2g-2$.

For a torus covering to lie in the minimal stratum, it is necessary 
that it be branched over a single point. 
To describe such coverings, consider the unbranched covers of the 
punctured torus $\Cbb/((\Zbb+ i \Zbb) \setminus\{0\})$. 
Each such degree $d$ covering corresponds to a homomorphism  $\rho$
from the fundamental group of the once punctured torus
to the symmetric group on $d$ letters (the `monodromy representation'). 
The fundamental group of the once punctured torus is freely generated by 
the central curve $h$ of the horizontal cylinder and the central
curve $v$ of the vertical cylinder. It follows that 
each degree $d$ covering that is branched over $0$ is determined by 
$\rho(h)$ and $\rho(v)$. In sum, each branched covering is determined 
by a pair of permutations that we will denote $h$ and $v$ respectively. 
This description is unique up to simultaneous conjugation of $h$ and $v$.

There is a one-to-one correspondence between 
the zeros of $p^*(dz)$ and the nontrivial cycles of the commutator $[h,v]$. 
Each cycle of length 1 in $[h,v]$ 
corresponds to a point in the fiber above $[0]$ that
is not ramified. In particular, since in this section,
we wish to construct torus coverings with a single ramification point
of degree $2g-1$ we will need to check that $[h,v]$ has one 
cycle of length $2g-1$ and $d-(2g-1)$ cycles of length 1. 

Torus coverings branched over one point are often 
called {\em square-tiled surfaces}. 
Indeed, given a pair of permutations $h$, $v$ of $\{1,\ldots, d\}$,
we can construct the covering by gluing together $d$ disjoint unit squares 
labeled $1, \ldots, d$ as follows: Glue the right side of square $i$ to the 
left side of square $h(i)$ and the top of square $i$ to the bottom of square $v(i)$.  
Note that the group generated by $h$ and $v$ must act transitively on $\{1, 2, ..., d\}$ 
for the surface to be connected.


\subsubsection{The hyperelliptic component}

Let $p\colon  H_g \to T$ be the degree $d = 2 g-1$ torus covering branched 
over one point that is
defined by the following permutations on $2g-1$ letters (in cycle notation)
\begin{eqnarray*}
h  &=& (1,2)(3,4)~ \cdots~ (2g-3,2g-2)(2g-1) \\
v  &=& (1)(2,3)(4,5)~ \cdots~ (2g-2,2g-1).
\end{eqnarray*}
See Figure \ref{fig:H_g}.
The commutator $[h,v]$ has order $2g-1$ and so $p$ has only
one ramification point, and thus $p^*(dz)$ has exactly 
one zero $z$ of order $2g-2$. Hence each vertical edge (resp. 
horizontal edge) of each unit square is a 1-cycle in 
$H_1(H_g;\Zbb)$, and the covering map sends this 1-cycle
to the standard vertical (resp. horizontal) generator
of $H_1(\Cbb/\Zbb^2; \Zbb)$. Hence $p$ is primitive.

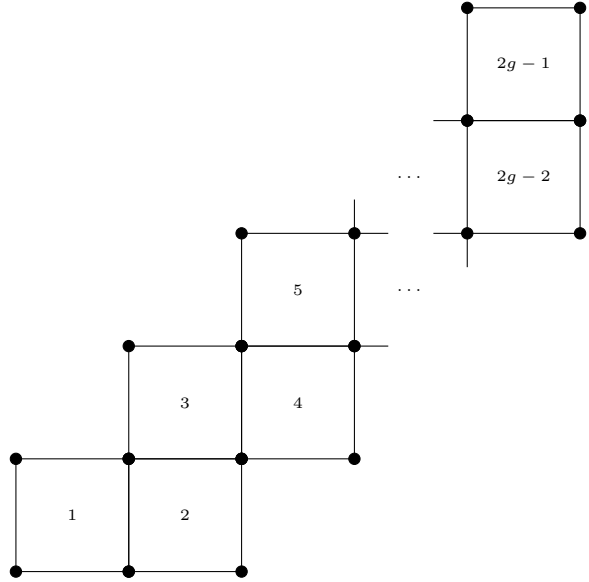
\begin{figure}[h!]
    \centering
    \begin{tikzpicture}[scale=\tikzscale,font=\tiny]
        \tile{0}{0}
        \node at (0.5,0.5) {$1$};
        \tile{1}{0}
        \node at (1.5,0.5) {$2$};
        \tile{1}{1}
        \node at (1.5,1.5) {$3$};
        \tile{2}{1}
        \node at (2.5,1.5) {$4$};
        \tile{2}{2}
        \node at (2.5,2.5) {$5$};
        \draw (3,2) -- (3.3,2);
        \draw (3.5,2.5) node{$\cdots$};
        \draw (3,3) -- (3.3,3);
        \draw (3,3) -- (3,3.3);
        \draw (3.5,3.5) node{$\cdots$};
        \draw (3.7,3) -- (4,3);
        \draw (4,2.7) -- (4,3);
        \tile{4}{3}
        \node at (4.5,3.5) {$2g-2$};
        \draw (3.7,4) -- (4,4);
        \tile{4}{4}
        \node at (4.5,4.5) {$2g-1$};
    \end{tikzpicture}
    
    \caption{The hyperelliptic torus cover, 
       $H_g$, in the minimal stratum that is a degree $2g-1$ primitive 
     branched covering of the torus.}
    \label{fig:H_g}

\end{figure}

A hyperelliptic involution $\tau$ of the surface $H_g$ can be constructed
by rotating each square in Figure \ref{fig:H_g} about its center by $\pi$
radians.
The involution $\tau$ has $2g+2$ fixed points  consisting
of the zero of $p^*(dz)$, the centers of 
each of the $2g-1$ squares, the midpoint of the top (and bottom) edge of 
square 1, and the midpoint of the left (and right) edge of square $2g-2$. 
The quotient $H_g/\langle \tau\rangle$ is a sphere 
and it follows that $p^*(dz)$ is hyperelliptic.

To construct primitive branched covers of degree $d>2g-1$, we lengthen 
one of the vertical cylinders by placing $d-(2g-1)$ additional squares
on top of the square $2g-1$ in Figure \ref{fig:H_g}.
To be precise, let
$p\colon  H_g^d \to \Cbb/\Zbb$ be the covering determined by the 
permutations
\begin{eqnarray*}
h &=& (1,2)(3,4)~ \cdots~ (2g-3,2g-2)(2g-1)(2g-2)~ \cdots~ (d-1)(d) \\
v  &=& (1)(2,3)(4,5)~ \cdots~ (2g-2,2g-1, \ldots, d-1,d ).
\end{eqnarray*}
The commutator $[h,v]$ has one cycle of length $2g-1$ and $d-(2g-1)$
cycles of length 1. In other words, $p^*(z)$ has a single zero of
order $2g-2$. The covering $p$ is primitive for the same reason that 
the covering $H_g \to T$ is primitive.
The surface $H_g^d$ admits a hyperelliptic involution $\tau$ which
rotates by $\pi$ each of the squares labeled $1$ through $2g - 2$ 
about their respective centers. The involution $\tau$ preserves the 
horizontal Euclidean cylinder $C$ consisting of the squares $2g-1, \ldots, d$,
and its restriction to $C$ has two fixed points. 
The only remaining fixed point of $\tau$ is the unique 
zero of $p^*(dz)$. 


\subsubsection{The odd component}
\label{subsubsec:odd}

Let $p\colon  O_g \to T$ be the degree $d = 2 g-1$ torus covering branched 
over one point that is
defined by the following permutations on $2g-1$ letters (in cycle notation)

\begin{eqnarray*}
h &=& (1,3,5, \ldots, 2g-1) \\
v  &=& (1,2)(3,4)~ \cdots~ (2g-3,2g-2)
\end{eqnarray*}
See Figure \ref{fig:O_g-square}.

\begin{figure} 
\includegraphics[scale=.8]{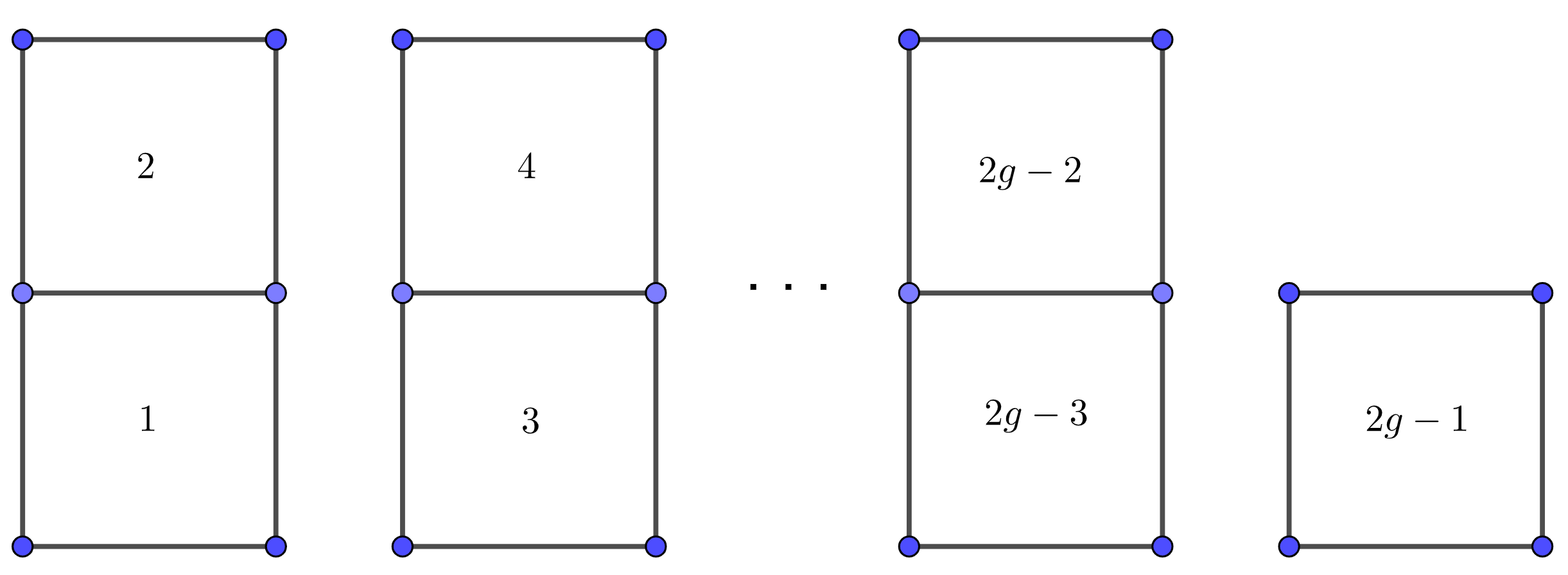}
\caption{The odd spin parity torus cover $O_g$ in the minimal 
stratum. The odd numbered squares together form a horizontal 
cylinder of length $g$ and each even numbered square
corresponds to a horizontal cylinder of length $1$.}

\label{fig:O_g-square}

\end{figure}

To see that the surface is not hyperelliptic, we suppose that $O_g$ 
admits a hyperelliptic involution $\tau$ and then derive 
a contradiction. The horizontal cylinder $C$ that consists of the odd
numbered squares is the only horizontal cylinder of length greater
than 1, and hence $\tau(C)=C$. In particular, the map $\tau$ 
preserves the union $V$ of the vertical saddle connections that are 
contained in $C$. The complement of $V$ consists of the 
vertical cylinder corresponding to the square labeled $2g-1$
and the $g-1$ slit tori $S_i$ corresponding to the cycles $(i, i+1)$ for $i$ odd. 
If $\tau(S_i)=S_j$ for some $i \neq j$, then the sphere $O_g/\langle \tau\rangle$
would contain a once holed torus. This is not possible, and so
$\tau(S_i)=S_i$ for each $i$. In particular, $\tau$ preserves 
each odd numbered square, and the center of each odd numbered square
is a fixed point. That is, $\tau$ has at least $g-1 >2$ fixed points.
But each (non-null homologous)
cylinder $C$ has exactly $2$ fixed points,
and this is the desired contradiction.

To see that the spin parity of $O_g$ is odd, we exhibit  
$O_g$ as the slit tori decomposition mentioned in the previous paragraph. 
See Figure \ref{fig:O_g-new-symplectic}. Here we have 
chosen a symplectic basis $\{a_i, b_i\}$ for the first homology
of $O_g$. The curve $a_g$ `turns' once as it traverses each slit torus
and hence has index equal to $g-1$.
All other curves in this symplectic basis are geodesics 
and hence have index equal to zero. 
Thus, it follows from formula (\ref{eqn:spin-parity}) 
that the spin parity equals $2g-1 \mod 2$. 

The map $p$ is primitive because, for example, 
the classes $p_*(a_1)$ and $p_*(b_g)$ generate the first homology 
of $\Cbb/ (\Zbb + i \Zbb)$.

\begin{figure} 

\includegraphics[scale=1.0]{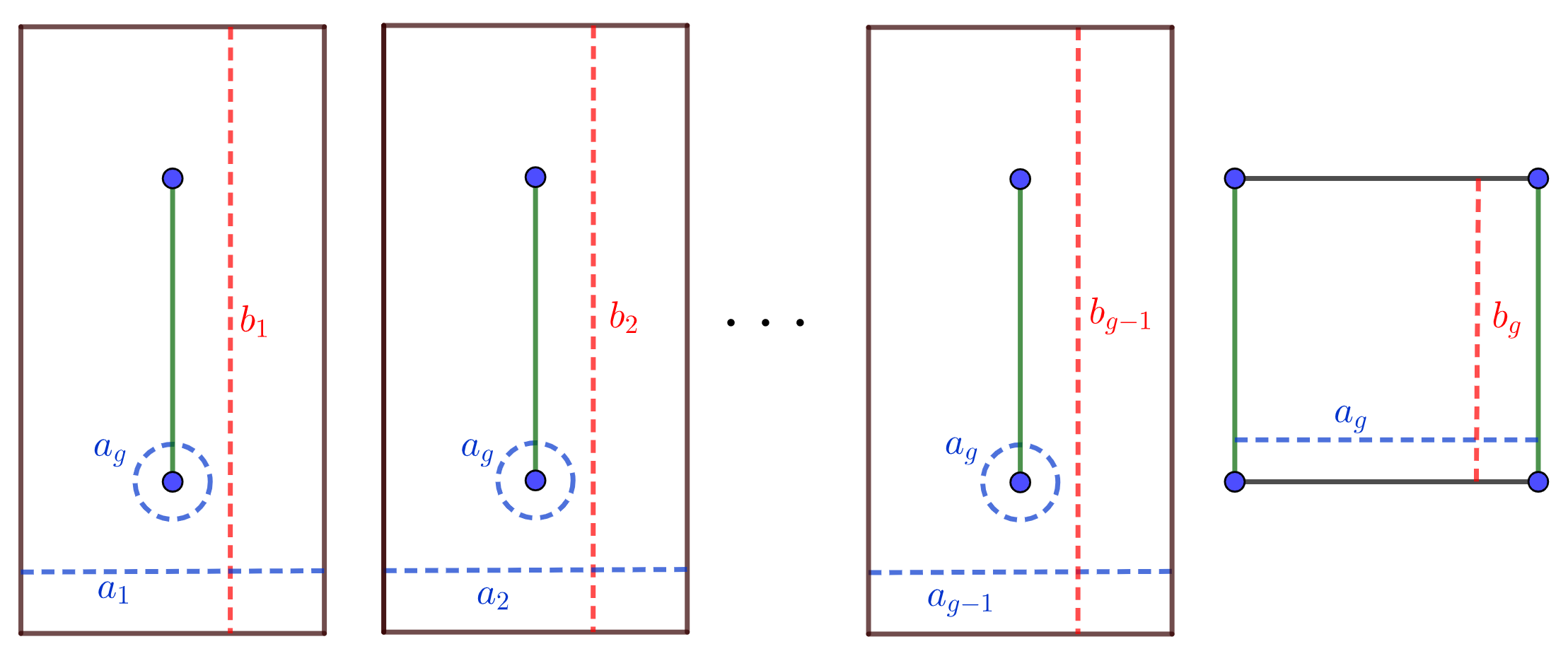}

\caption{
The simple closed curves $\{a_i,b_i\}$ form a symplectic basis 
for the first homology  of $O_g$. 
Note that the curve $a_g$ intersects each slit torus, and each
intersection contributes $1$ to the index of $a_g$.}

\label{fig:O_g-new-symplectic}

\end{figure}

To construct a primitive branched cover $p:O_g^d \to \Cbb/ (\Zbb + i \Zbb)$ 
of degree $2g-1+d$, replace the cycle $(2g-1)$ in the horizontal 
permutation $h$ with $(2g-1, 2g, \ldots, d-1,d)$. This is equivalent
to replacing the vertical cylinder that 
corresponds to the square labeled $2g-1$  with a vertical cylinder
of width $d-(2g-2)$.

\subsubsection{The even component}
\label{subsec:even-minimal}

For $g \geq 4$, let $p \colon E_g \to T$ be the degree $d=2g-1$ branched covering 
of $\Cbb/(\Zbb+ i \Zbb)$ defined by the permutations
\begin{eqnarray*}
h &=& (1,3,5,~\ldots~,2g-1,4) \\
v  &=& (1,2)(3,4)~ \cdots~ (2g-3,2g-2)(2g-1).
\end{eqnarray*}
See Figure \ref{fig:E_g-new}. The surface $E_g$ differs from $O_g$ in the way that 
the squares labeled 3 and 4 are attached. Arguments similar to the ones given 
in \S \ref{subsubsec:odd} show that $p$ is not hyperelliptic, is of even spin parity,
and is primitive. For example, the horizontal cylinder $C$ consisting of the square
labeled 4 and the odd numbered squares would be preserved by a hyperelliptic involution, 
and one can use this to argue that $E_g$ is not hyperelliptic. All of the 
elements in the symplectic basis in Figure \ref{fig:E_g-new} have index zero 
except for $a_2$ and $a_g$ which have indices $1$ and $g-1$ respectively.
In particular, the spin parity is $2g \mod 2$.

To obtain a degree $d$ cover $p : E_g^d \to \Cbb / (\Zbb + i \Zbb)$,
replace the vertical cylinder of width $1$ corresponding to the square 
labeled $2g - 1$ with a vertical cylinder of width $d - (2g - 2)$.
In other words, replace the cycle $(2g-1)$ that appears in $v$
with the cycle $(2g-1, 2g, \ldots, d-1,d)$.

\begin{figure} 

\includegraphics[scale=1]{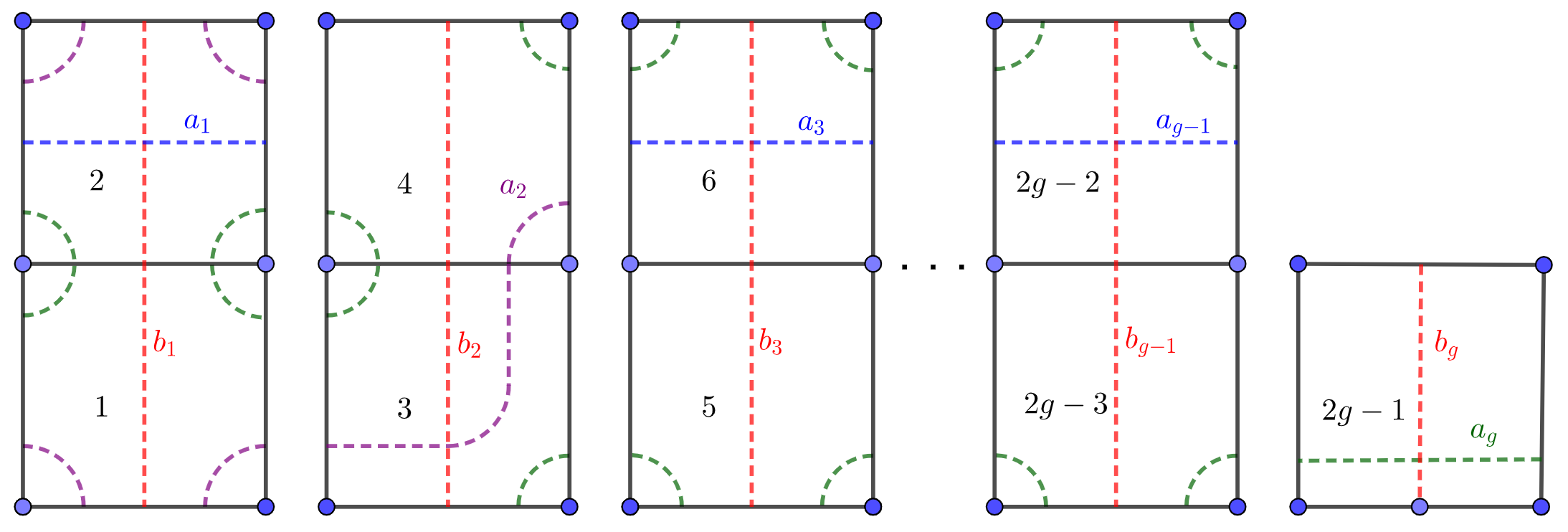}

\caption{
The even parity torus cover $E_g$ in the minimal stratum. 
The simple closed curves $\{a_i,b_i\}$ form a symplectic basis 
for the first homology  of $O_g$. 
The intersection of $a_g$ with each vertical cylinder of
height two contributes 1 to the index of $a_g$. All other basis
elements have index $0$ except for $a_2$ which has index $1$.}

\label{fig:E_g-new}

\end{figure}


\subsection{The strata with two zeros of equal order}  
\label{subsec:two-zeros}

According to \cite{KoZo}, if $g \geq 5$ is odd, then the stratum 
$\Omega \Mcal(g-1, g-1)$ has 
three connected components: hyperelliptic; even spin parity and non-hyperelliptic;
and odd parity and non-hyperelliptic. When $g=2,3$ or $g \geq 4$ and even,
the stratum has exactly two components: hyperelliptic and non-hyperelliptic.  
In \S \ref{subsec:g-1-hyp} we exhibit a surface in each hyperelliptic component, 
regardless of the parity of $g$, and then in \S \ref{subsec-g-1-non-hyp} 
we construct examples in the remaining non-hyperelliptic component(s).
Our constructions will be based on gluing together surfaces with slits.


\subsubsection{$\Omega \Mcal(g-1,g-1)^{\text{hyp}}$} \label{subsec:g-1-hyp}

If $g=2m$ is even, we construct a degree $g$ hyperelliptic torus cover as follows.  
First, create a genus two surface by gluing together two copies of $\Cbb/\Zbb^2$
that each have a horizontal slit. Take $m$ distinct copies, $S_1, \ldots, S_m$,
of this genus two surface.  From both $S_1$ and $S_m$ remove 
one of the two horizontal saddle connections that are distinct from slits
and from each of the remaining genus two surfaces, $S_2, \ldots, S_{m-1}$,
remove both of these horizontal saddle connections. Glue the top 
(resp. bottom) of the new slit on $S_1$ to the bottom (resp. top) of one
of the (new) slits on $S_2$. Then, inductively, glue the top (resp. bottom) of 
the remaining slit on $S_i$ to the bottom (resp. top) of one
of the slits on $S_{i+1}$. Let $X_g$ denote the resulting degree $g$ 
cover of $\Cbb/(\Zbb+ i \Zbb)$ when $g$ is even.

\begin{figure}

    \includegraphics[scale=.7]{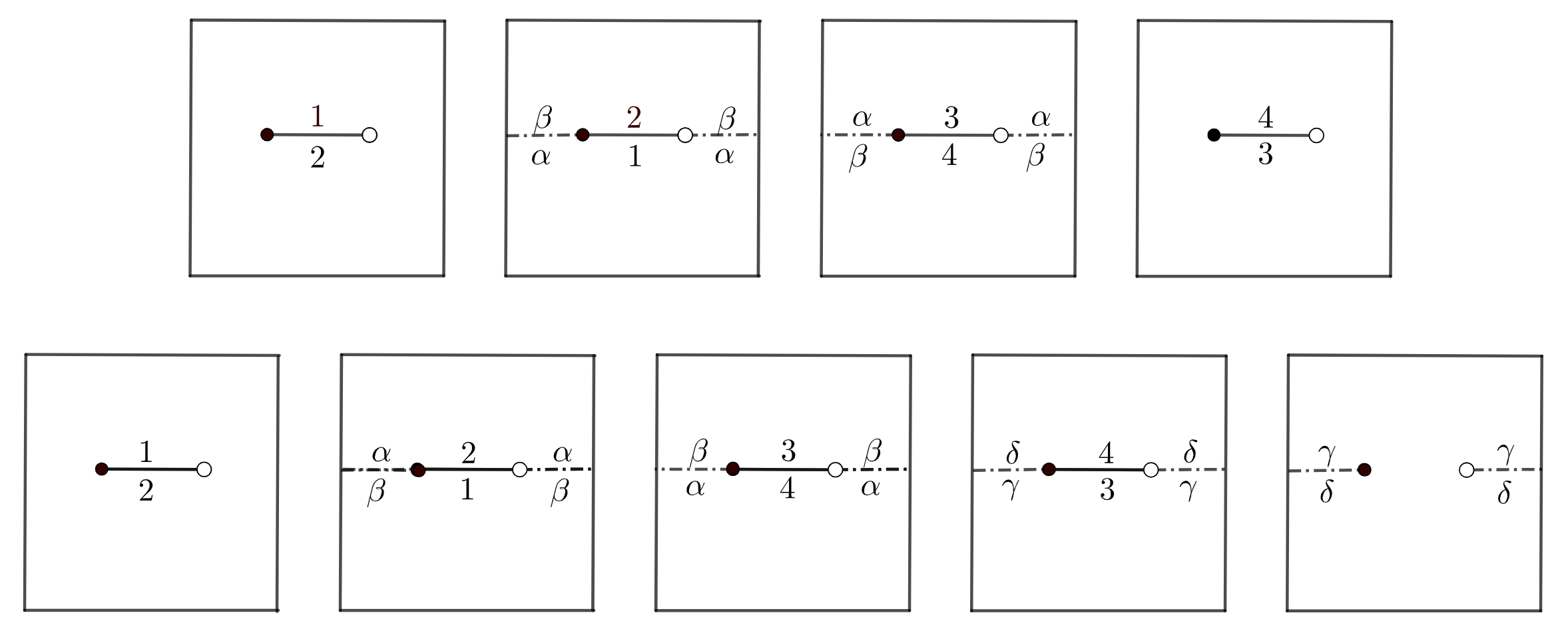}

    \caption{Primitive degree $g$ torus covers  in $\Omega\mathcal{M}_g(g-1, g-1)^{\text{hyp}}$ in the cases $g=4$ and $g=5$. Each square corresponds to a slit torus.}
    \label{fig:hyp(g-1,g-1)}
\end{figure}

If $g=2m +1$ is odd, then remove the horizontal saddle connection of $X_{2m}$
that lies in $S_m$ and then glue in an additional horizontally slit torus
to obtain the torus cover $X_{2m+1}$. The surfaces $X_4$ and $X_5$  
are described in Figure \ref{fig:hyp(g-1,g-1)}.

A torus cover $X_g^d$ of degree $d= k+ g-1$ can 
be constructed in the same way if one replaces a slit 
copy of $\Cbb/(\Zbb+ i\Zbb)$ in the construction of the genus two 
surface $S_1$ with a slit copy of $\Cbb/ (k \Zbb + i \Zbb)$.
The hyperelliptic involution on $X_g^d$ corresponds to the 
elliptic involution of each slit torus that fixes the center 
of each slit.  A vertical curve in $S_1$ 
(resp. horizontal curve in $S_2$) is mapped to the 
standard vertical (resp. horizontal) generator of $H_1(\Cbb/ (\Zbb+i \Zbb),\Zbb)$.
Hence the covering is primitive.

\begin{remk}
A degree $g$, primitive, hyperelliptic torus covering can also be 
defined in terms of the classical Chebyshev polynomial $P_g$, 
the unique polynomial satisfying
\[
    P_g(\cos \theta)~ =~ \cos \left(g \cdot \theta \right)
\]
for each $\theta \in \Rbb$.
Given $a \in (0,1)$ such that $P_g(a)\neq \pm 1$,
let $q$ be the unique quadratic differential on the Riemann sphere 
$\widehat{\Cbb}$ with simple poles at $\{\pm 1,\pm a\}$. 
The set $P_g^{-1}\{+1, -1\}$ consists of all $g-1$ critical points of degree
two together with the two additional points at which $P_g$ is not branched. 
The map $P_g$ is not branched at any of the $2g$ points in $P_g^{-1}\{+a, -a\}$. 
It follows that $P_g^*(q)$ has $2+2g$ simple poles and 
one zero of degree $2g-2$ at $\infty \in \widehat{\Cbb}$. 
Let $(X,\omega)$ and $(\Cbb/\Lambda,dz)$ be the respective 
canonical double covers of $(\widehat{\Cbb},P_g^*(q))$ and $(\widehat{\Cbb}, q)$. 
The map $P_g\colon \widehat{\Cbb} \to \widehat{\Cbb}$ lifts 
to a primitive degree $g$ branched cover 
$\tP_g \colon X \to \Cbb/\Lambda$ so that $\tP_g^*(dz)=\omega$. 
It follows that $(X, \omega)$ lies in the 
hyperelliptic component of $\Omega \Mcal(g-1,g-1)$. 
\end{remk}


\subsubsection{Non-hyperelliptic components of $\Omega \mathcal{M}_g(g-1,g-1)$}
\label{subsec-g-1-non-hyp}

Recall that if $g=3$ or $g \geq 4$ and $g$ is even, then there 
is exactly one non-hyperelliptic component. 
If $g \geq 5$ and $g$ is odd, then there are exactly 
two non-hyperelliptic components, 
one consisting of odd spin parity 1-forms and 
one consisting of even spin parity 1-forms.
We first construct a torus covering that is non-hyperelliptic in each genus
and then observe that if $g$ is odd, then its spin parity is odd.
Then we separately construct an even spin torus covering for $g$ odd. 

For each $g \geq 3$, define a degree $g$ torus cover 
$X_g$ by cyclically 
gluing together distinct horizontally slit tori $S_1, \ldots, S_g$. 
To be more precise, glue the top of the slit on $S_i$ to the bottom
of the slit on $S_{i+1}$. The case of $g=5$ is illustrated in Figure 
\ref{fig:non-hyp odd}.

\begin{figure} 
    \includegraphics[scale=.7]{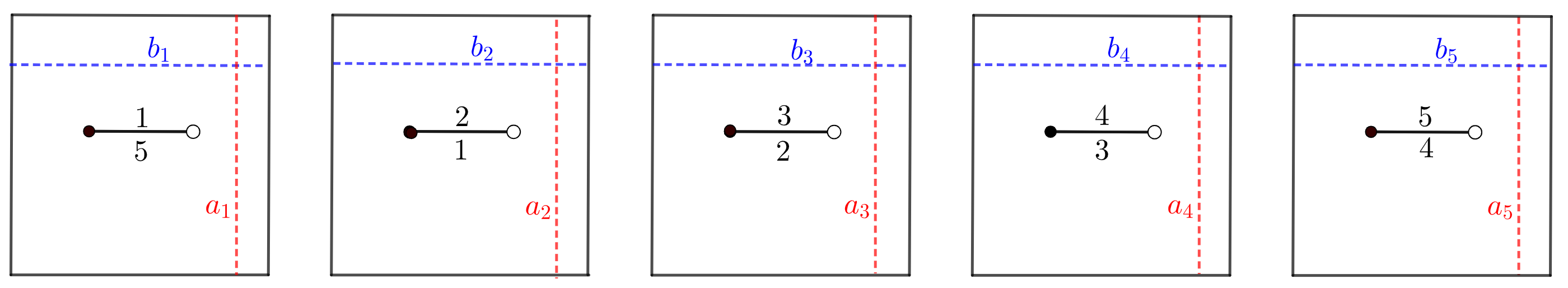}
    \caption{A cyclically glued $g$-slit torus cover $X_g$ when $g=5$}
    \label{fig:non-hyp odd}
\end{figure}

To prove that the surface $X_g$ is not hyperelliptic, let us assume
to the contrary that a hyperelliptic involution $\tau$ exists and derive
a contradiction.
Let $C$ be the vertical cylinder that contains each of the slits $s_i \subset S_i$. 
The cylinder $C$ is the only vertical cylinder that has length greater than one, 
and hence it would be preserved by a hyperelliptic involution $\tau$.
Thus, $\tau$ would preserve the union of horizontal saddle connections
that belong to $C$, and hence would preserve the complement $A$, that is the 
disjoint union of the slit tori $S_i$. If $\tau$ were to map one slit torus $S_i$ onto
a distinct slit torus $S_j$, then the quotient $X_g/\langle \tau\rangle$ 
would contain the embedded one-holed torus $S_i \cup S_j/\langle \tau\rangle$, 
and hence the quotient would not be a sphere. 
Thus the hyperelliptic involution $\tau$ would have to preserve each $S_i$, and hence
would act as an elliptic involution on each $S_i$. It follows that 
the involution $\tau|_{S_{i}}$ has a fixed point $x_i \in C$. Hence $C$ 
contains $g$ fixed 
points, and since $g \geq 3$, this is the desired contradiction.

When $g$ is odd, then the spin parity of $X_g$ is well-defined, and  
a straightforward argument shows that the spin parity of $X_g$ is odd. 
Indeed, choose a homology basis for each slit torus $S_i$ consisting
of a vertical and a horizontal curve. The index of each of these curves is zero. 
Thus, the spin parity of $X_g$ is $\sum_{i=1}^g 1 \equiv  g \mod 2$.

To obtain non-hyperelliptic covers $p \colon X_g^d \to T$ of degree $d=g-1+k$, one may modify 
the construction by replacing, for example, $S_1$ with the slit torus obtained by
removing a horizontal slit $s$ from the torus $\Cbb/(k \Zbb + i\Zbb)$. 
Similar arguments show that $X_g^d$ is not hyperelliptic and has spin parity equal
to $g \mod 2$.

It remains to construct, for each odd $g\geq 5$ and each $d \geq g$,
a non-hyperelliptic, even spin parity, torus cover in $\Omega \Mcal(g-1,g-1)$
of degree $d$. To construct it for degree $d=g$, we start from the surface $X_g$ 
with odd parity, cut along two slits, and glue differently. To be precise, 
recall that $X_g$ is the union of slit tori $S_1, \ldots S_g$.
For each $i$, the slit of $S_i$ is a segment in a horizontal closed geodesic
in the torus. Let $\delta_i$ be the complementary segment in this
horizontal geodesic. In particular, the corresponding segment on the 
surface $X_g$ is a horizontal saddle connection that we still denote by $\delta_i$. 
Remove $\delta_2$ and $\delta_{g-2}$ from $X_g$. 
Glue the top (resp. bottom) of $\delta_2$ to the bottom (resp. top)
of $\delta_{g-2}$. See Figure  \ref{fig:non-hyp-even}
for the case of $g=5$. The top of the slit $\delta_2$ and the bottom of the 
slit $\delta_{g-2}$ are labeled with $\alpha$, and the bottom of $\delta_2$ 
and the top of $\delta_{g-2}$ 
are labeled by $\beta$.  The resulting surface 
$Y_g$ covers $T$, and using the homology basis illustrated in Figure
\ref{fig:non-hyp-even},
one finds that the spin parity is $g-2 +2 +3 \equiv g+3 \mod 2$.

To obtain torus covers of higher degree one need only, as above, replace one of 
the slit tori with the slit torus coming from $\Cbb/(k\Zbb + i\Zbb)$.  
To see that the surface $Y_g$ is not hyperelliptic, apply the argument 
used for $X_g$ to the unique vertical
cylinder $C$ in $X_{g-2}$ that has circumference greater than $2$.

\begin{figure}
    \includegraphics[scale=.7]{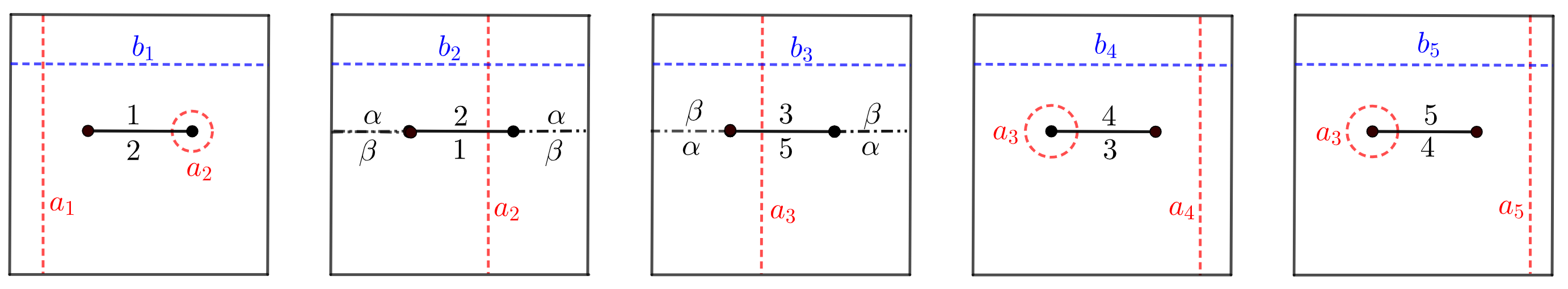}
    \caption{A genus 5 torus covering $Y_g$ in the 
     even spin parity component of $\Omega \mathcal{M}(4,4)$.
    } 
    \label{fig:non-hyp-even}
\end{figure}


\subsection{Surgeries that add zeros and preserve degree}   
\label{subsec:other}

Thus far, we have produced torus coverings in each connected component of both
the minimal stratum $\Omega \Mcal_g(2g-2)$ and the stratum
$\Omega \Mcal_g(g-1,g-1)$. To obtain torus covers in all connected 
components of all other strata, we will perform certain `surgeries' 
on the torus covers $E_g^d$ and $O_g^d$ in the minimal strata 
as well as a variant, $Z_g^d$, of these that will be described 
in \S \ref{subsec:highest-order-odd}.
Each surgery described here modifies the torus covering by adding zeros, 
increasing genus, and preserving degree. 
Each surgery can be performed on a torus cover branched over one point 
that has at least one vertical cylinder of circumference one and that has
sufficiently many vertical cylinders of circumference at least two.

To be precise, let $p: X \to T$ be a torus covering
of degree $d$ such that there exists a  
vertical cylinder $C \subset T$ that
does not contain a branch point of $p$.\footnote{The boundary of the cylinder $C$ need not contain any branch points.}
We will say that the torus covering $p$ is {\it surgery admissible 
with respect to $C$ and $k$} if the components of $p^{-1}(C)$ consist of
\begin{itemize}
\item at least $k$ cylinders each having circumferences at least two, and

\item  at least one nonseparating cylinder whose circumference equals one.
\end{itemize}
In particular, to be surgery admissible $p$ must have degree $d \geq 2k+1$.

In \S \ref{subsec:single-even-order}, we show how to add a zero 
of order $2k$
to a surgery admissible covering, and in \S \ref{subsubsec:pair_same_odd}
we show how to add a pair of odd order zeros. Each of these surgeries produces 
a surgery admissible torus covering. Therefore, we may apply any finite sequence 
of these surgeries.
In \S \ref{subsec:parity-comp}, we show how to calculate the  
change in the spin parity so as to be sure that we can obtain a torus
covering in each connected component of a stratum.

According to Theorem 1 in \cite{KoZo}, components consisting of
hyperelliptic surfaces only occur in the strata $\Omega \Mcal(2g-2)$ and 
$\Omega \Mcal(g-1,g-1)$. Thus, we will 
not need to consider the effect 
of surgeries on hyperellipticity.


\subsubsection{Adding a zero of order $2k$} 
\label{subsec:single-even-order}
Let $p: X \to T$ be a surgery admissible torus cover.
In this subsection, we describe a `surgery' on this torus covering that 
yields a surgery admissible torus 
covering $\op: \oX \to T$ with the same degree 
$d$ and an additional zero of order $2k$. 

Let $C \subset T$ be the vertical cylinder
that does not contain a branch point of $p$. 
Let $C_0$ be a component of $p^{-1}(C)$ 
that has circumference 1 and let  $C_1, \ldots, C_k$ 
be components of circumference at least two. 
Choose a vertical closed geodesic $\sigma \subset C$ and 
choose $P \in \sigma$.
The inverse image $p^{-1}(\sigma)$ consists of disjoint closed geodesics. 
The set $p^{-1}(\sigma \setminus \{P\})$ consists of $d$ disjoint vertical segments. 
If the vertical cylinder $C_i$ has circumference $m$, then 
at least $m$ of these segments lies in $C_i$. 
Choose exactly one segment $\sigma_i$ from each of the cylinder $C_i$. 
See Figure \ref{fig:even_order}.
Cut along each $\sigma_i$ and glue the left side of $\sigma_i$ to the right side
of $\sigma_{i+1}$. Let $\oX$ be the resulting surface.

The covering $p$ determines a surgery admissible 
torus covering $\op: \oX \to T$
of degree $d$ that is branched over $0$ and $P$.
Moreover, the 1-form $\op^*(dz)$ has an additional zero of order $2k$, 
and the genus of $\oX$ is $k$ greater than the genus of $X$.

\begin{figure}
    \centering
    \includegraphics[scale=1]{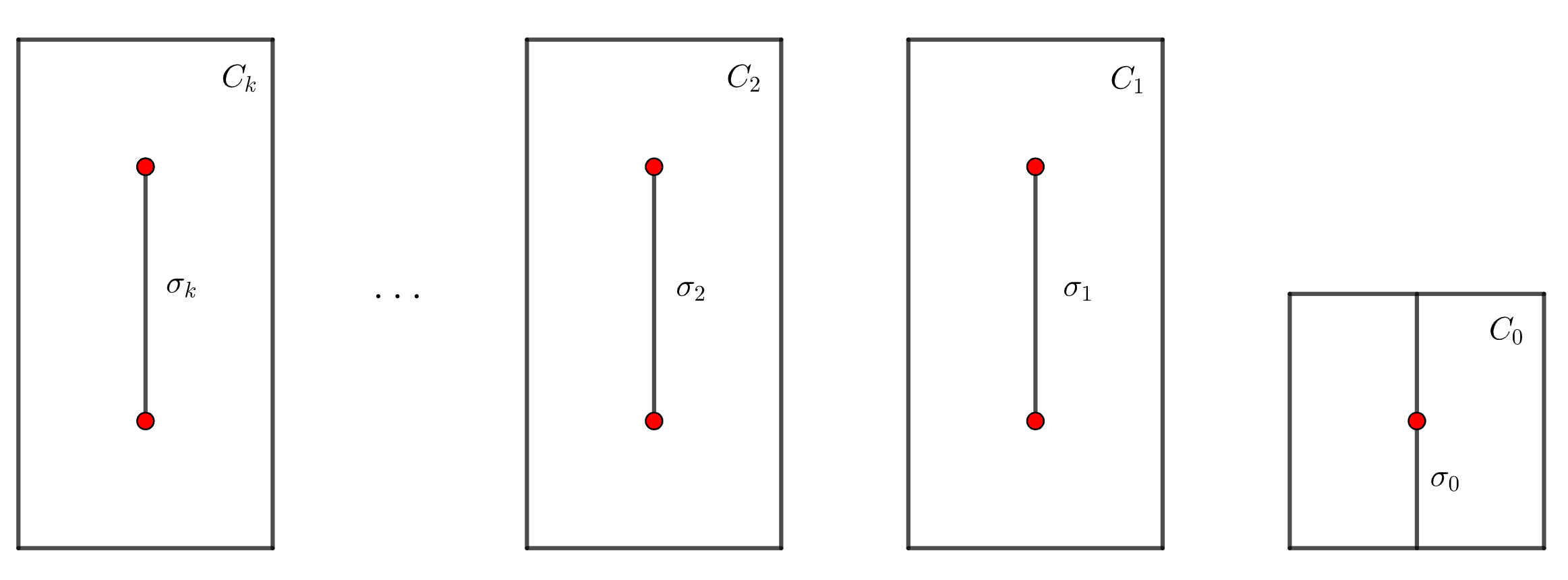}    
    \caption{Adding a zero of order $2k$. Cut along 
    each $\sigma_i$ and identify the left side of $\sigma_i$
    to the right side of $\sigma_j$. The red endpoints 
    are thus all identified with one another and 
    they represent a ramification point of local 
    index $2k+1$ over $P$. }
    \label{fig:even_order}
\end{figure}


\subsubsection{Adding a pair of zeros of odd order}
\label{subsubsec:pair_same_odd}

In this subsection we describe a surgery on $p: X \to T$
that adds a zero of order $2k-1$ and a zero of order $2k'-1$ where 
$k' \leq k$.\footnote{The orders of zeros of a holomorphic 1-form $\omega$ on a genus 
$g$ Riemann surface must sum to $2g - 2$, and so $\omega$ 
has even number of zeros of odd order. 
Thus, any surgery that increases the number of odd order
zeros will necessarily increase the number of odd order zeros by an even integer.} 
We will first assume that $k'=k$ and  
then show how to modify this surgery when $k' < k$.

To add a pair of zeros that have the same order $2k-1$,
choose a horizontal segment $\tau$ that lies in $C$
and has length strictly less than the width of $C$. Let $\tau_0$ 
be the unique component of $p^{-1}(\tau)$ that lies in $C_0$.
For each $i \in \{1,\ldots, k-1\}$, choose two connected components, 
$\tau_{2i-1}$ and $\tau_{2i}$, of $p^{-1}(\tau)$ that lie in $C_i$, 
and choose one component, $\tau_{2k-1}$,
of $p^{-1}(\tau)$ that lies in $C_k$. See Figure \ref{fig:odd-same}.
Cut the surface $X$ along each $\tau_i$. Then 
identify the top of $\tau_i$ with the bottom of $\tau_{i+1}$. The resulting 
surface is a degree $d$ torus cover with two new zeros of order $2k-1$.

\begin{figure}
    \centering
    \includegraphics[scale=1]{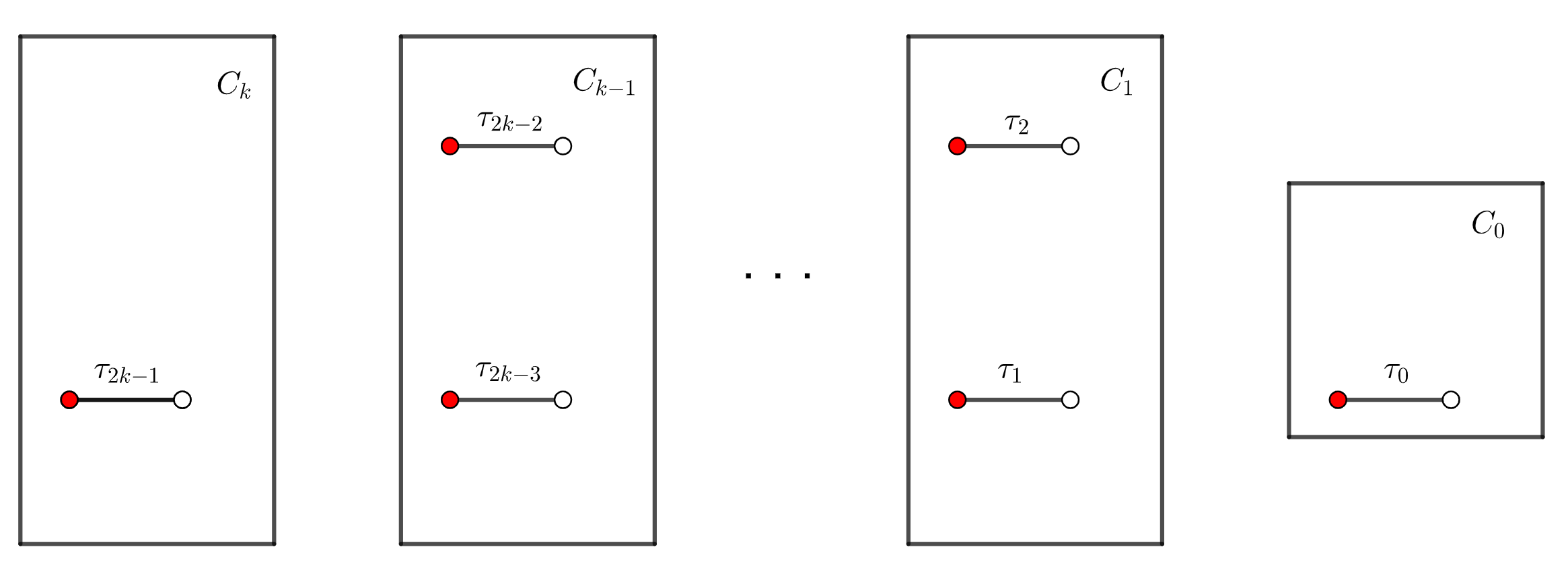}    
    \caption{The 
    white points are ramification points
    over one endpoint of $\tau$,
    and the red points are ramifications over the other endpoint of $\tau$.}
    \label{fig:odd-same}
\end{figure}

We next modify the construction to show how to add one zero of order 
$2k-1$  and one zero of order $2k'-1$ where $k'<k$. Roughly speaking, 
the surgery is a combination of the surgery that adds two zeros of order $2k'-1$ 
and the surgery that adds a zero of order $2(k-k')$. To be precise, let $\sigma$ be 
a vertical closed geodesic that lies in $C$
and let $\tau$ be a horizontal segment in $C$ 
that has one endpoint $P$ on $\sigma$. Let $\tau_0$ be the component of 
$p^{-1}(\tau)$ that lies in $C_0$ and let $\sigma_0$ 
be the lift of $\sigma$ to $C_0$.
For $i \in \{1,\ldots,k'-1\}$ choose two components, 
$\tau_{2i-1}$ and $\tau_{2i}$,
of $p^{-1}(\tau)$ that lie in $C_i$, choose one component, $\tau_{2k'-1}$,
of $p^{-1}(\tau)$ 
that lies in $C_{k'}$. For $i \in \{k'+1, \ldots, k \}$ 
choose one component, $\sigma_{i}$, of $p^{-1}\left(\sigma \setminus \{P\} \right)$
that lies in $C_i$. Cut along each  $\sigma_i$ and each $\tau_i$, 
cyclically reglue the $\sigma_i$, and cyclically reglue the 
$\tau_i$.  The new zero that corresponds to the point $P$ has order $2k-1$ 
and the new zero that corresponds to the other endpoint, $Q$, of $\tau$ has order $2k'-1$. 
See Figure \ref{fig:odd-different}  for an example of this construction.

\begin{figure}
    \centering
    \includegraphics[scale=1]{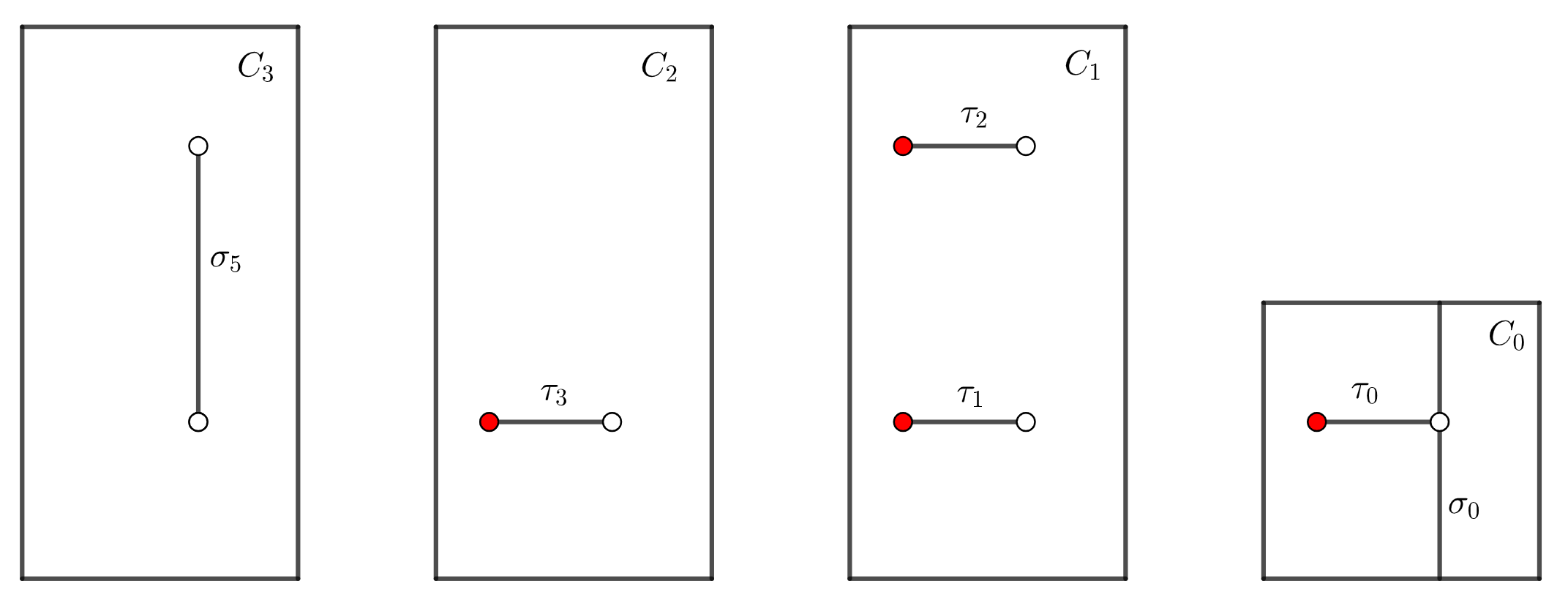}    
    \caption{Adding zeros of different odd orders to $X$. 
     The new zero that corresponds to the white points has order 5, and the 
     zero corresponds to the red points has order 3.}
    \label{fig:odd-different}
\end{figure}


\subsubsection{Change of parity computations} 
\label{subsec:parity-comp}

In this subsection we consider how the spin parity changes 
when the surgery described in \S \ref{subsec:single-even-order}
is applied. In particular, we find that 
adding a zero of order $2k$ 
preserves the spin parity if $k$ is even and it changes the spin parity
if $k$ is odd. Recall that the spin parity is not defined for 1-forms with 
zeros of odd order, and hence we will not consider the surgery
of \S \ref{subsubsec:pair_same_odd}.

Let $p: X \to T$ be a surgery admissible torus covering,
and let $\op: \oX \to T$ be the result of applying the surgery of
\S \ref{subsec:single-even-order}.

\begin{lem}
\label{lemma:add_order}
If $k$ is even, then the spin parity of $p^*(dz)$ equals the 
spin parity of $\op^*(dz)$. If $k$ is odd, then the spin parity 
of $p^*(dz)$ does not equal the spin parity of $\op^*(dz)$. 
\end{lem}

\begin{proof}
Let $C, C_0, C_1, \ldots, C_k$,  
$\sigma, \sigma_0, \sigma_1, \ldots, \sigma_{k}$, and $P$ be as in 
\S \ref{subsec:single-even-order}.

We first prove the statement in the special case of $k=1$. 
Let $b$ be a vertical closed geodesic in $C$ that is disjoint from 
$\sigma$ and let $b_0$ be the component of $p^{-1}(b)$ that lies in $C_0$. 
Since $X$ is surgery admissible, the simple closed curve $b_0$
is not null-homologous. Let $a_0$ be a simple closed curve 
on $X$ so that the geometric intersection number $i(a_0, b_0) = 1$,
so that $a_0$ does not intersect a ramification point of $p$,
and so that $a_0 \cap C_0$ is a horizontal segment.
We further suppose that $a_1$ intersects $\sigma_1$ orthogonally
at a point in $p^{-1}(a \cap \sigma)$. Thus, after cutting along 
$\sigma_0$ and $\sigma_1$ and regluing as  
described in \S \ref{subsec:single-even-order}, 
the closed curve $a_0$ becomes two simple closed curves, $a_0^+$ and $a_0^-$.
Let $a_0^+$ be the resulting simple closed curve that intersects 
$b_0^+ := b_0$ and let $a_0^-$ be the other curve.  
Let $b_0^-$ be a vertical geodesic in $C_0$ that intersects $a_0^-$.
See Figure \ref{fig:even-surgery-symplectic}.

\begin{figure}
    \centering
    \includegraphics[scale=2]{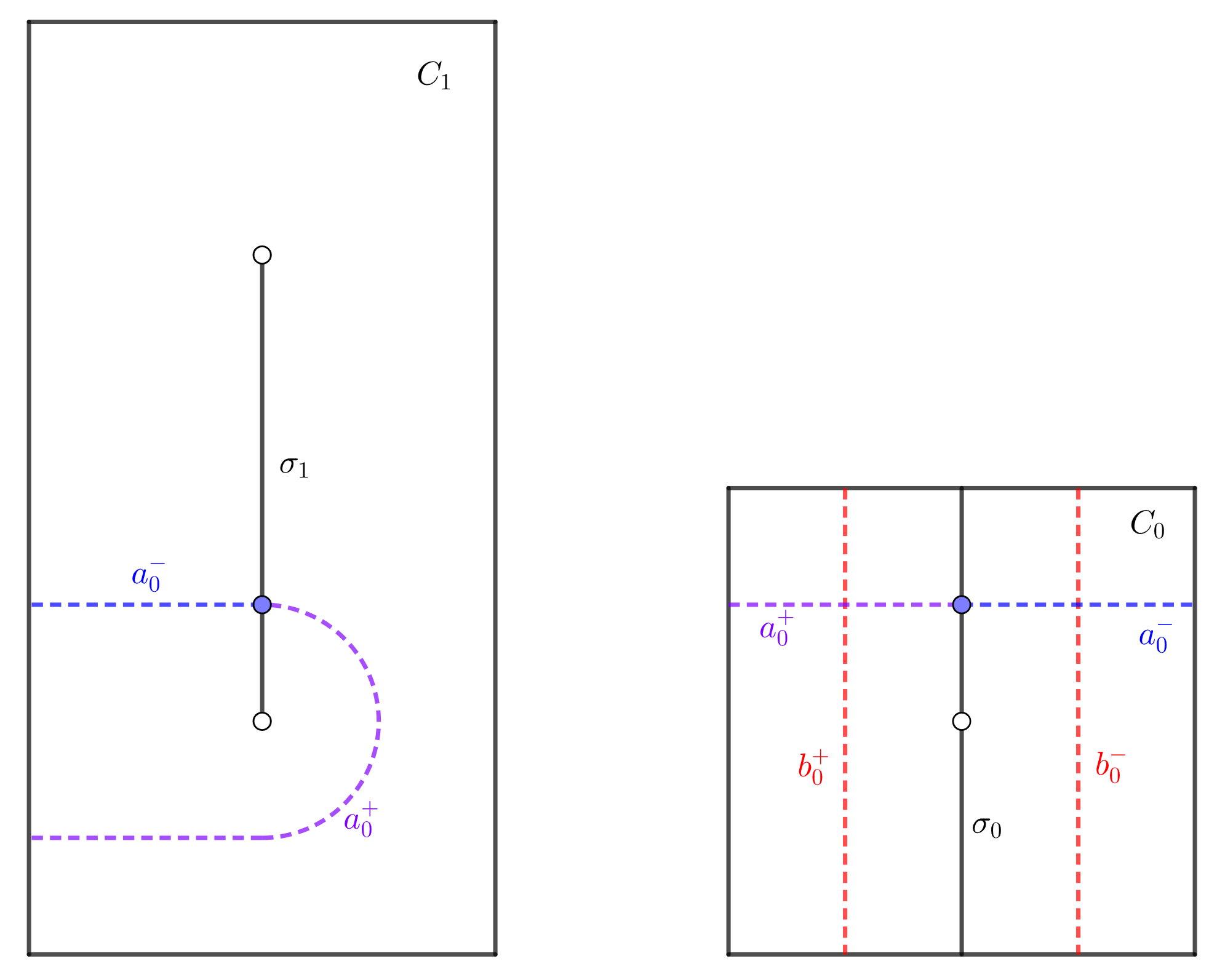}    
    \caption{The first four elements of symplectic basis for the
    surface that results from adding one zero of order two.}
    \label{fig:even-surgery-symplectic}
\end{figure}

Complete $\{a_0,b_0\}$ to a symplectic basis 
$\{a_0,b_0, \ldots, a_{g-1},b_{g-1} \}$ 
for $H_1(X; \Zbb)$ so that no $a_i$ nor $b_i$ intersects
a ramification point or $\sigma_1$ if $i >0$.
Then the collection 
$\{ a_0^+, b_0^+, a_0^-, b_0^-, a_1,b_1, \ldots, a_{g-1},b_{g-1} \}$
is a symplectic basis for the surface that results from the surgery. \
The curves $a_i$ and $b_i$ do not change if $i>0$ and hence their
indices do not change. The curves $b_0=b_0^+$ and $b_0^-$ are geodesics
and hence their indices equal zero. The index of $a_0$ equals the 
sum of the indices of $a_0^+$ and $a_0^-$. It follows that 
the spin parity `increases' by $1$. Hence the claim is proven in the 
case $k=1$.

To prove the claim for $k>1$, we will consider, for $i \leq k$, 
the result  $\op_i: \oX_{i} \to T$ of adding a zero of order $2i$ using
$C_0, C_1, \ldots, C_{i}$ and curves $\sigma_{0}, \ldots, \sigma_{i}$,
and we will consider the result  $\oq_i: \oY_i\to T$ 
of adding a zero of order $2$ to $\oX_{i-1}$. 
It suffices to show that for each $i \leq k$, 
the spin parity of $\op_i^*(dz)$ equals 
the spin parity of $\oq_i^*(dz)$. Indeed, an inductive
argument using the case $k=1$ would then imply the claim.

To prove that the spin parity $\op_i^*(dz)$ equals the spin parity of $\oq_i^*(dz)$,
we realize $\oY_i$ as an arbitrarily small perturbation of $\oX_i$. 
In particular, we choose $\delta>0$ and add a zero of order two to 
$\oX_{i-1}$ as follows: Let  $\sigma' \subset C$ be the vertical geodesic
to the `right' of  $\sigma$ such that the distance between 
$\sigma$ and $\sigma'$  equals $\delta$. Let $\alpha \subset T$ 
be the horizontal geodesic that intersects $\sigma$ at $P$.
Let $P'$ denote the intersection point of $\alpha$ and $\sigma'$.
Let $\sigma_{0}'$ be the component of $p^{-1}(\sigma') \cap C_0$ that
has distance $\delta$ from $\sigma_{0}$ and that is to the `right' of $\sigma_0$. 
Let $R_i$ be the connected component of $p^{-1}(C \setminus \alpha) \cap C_i$
that contains $\sigma_i$, and let $\sigma_{0}'$ 
be the component of $p^{-1}(\sigma' \setminus P)$ contained in $R_i$ 
that lies to the right of $\sigma_i$ and has distance $\delta$ from $\sigma_i$.
Cut along $\sigma_0'$ and $\sigma_1'$ and glue the left (resp. right) side of 
$\sigma_0'$ to the right (resp. left) side of $\sigma_1'$.
The resulting torus covering is $\oq_i : \oY_i \to T$. 

We next construct a piecewise differentiable homeomorphism 
$f: \oY_i \to \oX_{i}$ as follows. 
Define $f$ to be the identity on the complement of $C_0 \cup C_{i}$.
The right side of the segment $\sigma_{i}$ corresponds to a simple closed 
curve $\gamma \subset \oX_{i}$. Let $A$ be the annular neighborhood 
of $\gamma$ consisting of points of distance at most $\delta/2$ from $\gamma$.
Let $A^+$ be the connected component of $A \setminus \gamma$ that 
lies in $C_0 \cup C_{i}$. Define $f$ so that it maps the annulus 
$A' \subset C_0$ bounded by $\sigma_{0}$ and the right by 
$\sigma_{0}'$ onto the annulus $A_+$. Define $f$ so that it maps
the cylinder in $C_0$ that lies to the right of $\sigma'_0$ to
the part of the cylinder in $C_0$ that lies to the right of $\sigma_0$
that is exterior to $A^+$. Define $f$ to map the thrice holed
sphere $C_i \setminus \sigma_i'$ to the thrice holed sphere 
$C_i \setminus A^+$.  See Figure \ref{fig:even-surgery-homeo}.

\begin{figure}
    \includegraphics[scale=.32]{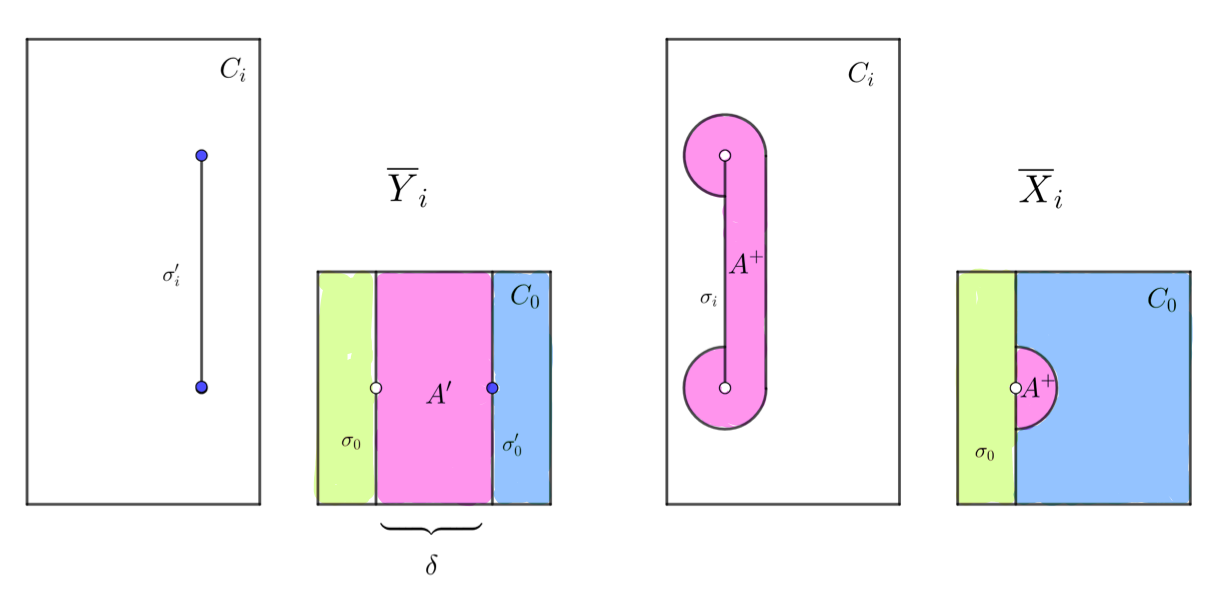}    
    \caption{The homeomorphism $f$ that maps $\oY_i$ to $\oX_i$.
    The map $f$ is the identity on the complement of 
    the cylinders $C_0$ and $C_i$. Each colored region
    in $C_0 \cup C_i$ is mapped to the region of the 
    same color in $C_0 \cup C_i$.}
    \label{fig:even-surgery-homeo}
\end{figure}

The construction of $f$ can be made to depend 
continuously on the parameter $\delta$. 
By pulling back the $1$-form $\oq^*(dz)$ using the inverse $f^{-1}$
we obtain a continuous family of $1$-forms $\omega_{\delta}$ on $\oX_i$.
Each zero of $\omega_{\delta}$ defines a  
simple arc on $\oX_i$ that is parametrized by $\delta$.
These arcs are disjoint, and hence we may choose a symplectic basis 
for $H_1(\oX_i; \Zbb)$ that avoids these arcs.  
It follows that the spin parity of $\omega_{\delta}$ 
is constant in $\delta$. Thus, for each $\delta$, the spin parity
of $\oq^*(dz)$ equals the spin parity of $\op^*(\delta)$.
\end{proof}


\subsection{Surgery admissible torus covers with highest order odd}  
\label{subsec:highest-order-odd}

In the next subsection, we will describe an algorithm which produces 
a degree $d$ torus cover in a prescribed connected component of a 
stratum that satisfies the hypotheses of Proposition \ref{prop:covering}.
The algorithm will be based on the surgeries described above. 
If the highest order of a zero in the prescribed stratum is even, 
then we will choose the initial surface to be either the
torus covering $E_g^d$ or the torus covering $O_g^d$ (see \S \ref{subsec:minimal-strata})
depending on the desired spin parity. In this section we construct 
a surgery admissible degree $d$ torus covering $p \colon Z^d_{m,n} \to T$
which will be the starting point of the algorithm when the zero of highest order in 
the prescribed stratum is odd.  The surface $Z^d_{m,n}$ will have genus $m+n$ 
and the associated $1$-form $p^*(dz)$ 
will have exactly two zeros, one with order $2m-1$ and the 
other with order $2n-1$. We will assume that $m \geq n$.

We will construct $Z^d_{m,n}$ using $O_m^d$ though
one could equally well use $E_m^d$.  Let $T$ denote the square torus 
$\Cbb/(\Zbb+ i \Zbb)$. The map $p \colon O_m^d \to T$ naturally extends 
to a degree $2m$ branched covering map, $\tp$, from the disjoint union of $O_m^d$ 
and $T$ to the square torus $T$, by defining the restriction of $\tp$ to $T$
to be the identity map.

Choose a horizontal segment $\tau$ in  $T$ of length $\epsilon< 1$ 
that has one endpoint at the origin in $T$. The inverse image $\tp^{-1}(\tau)$
has $2m(\geq 2n)$ connected components. Let $\tau_{2n}$ denote the unique
connected component of $\tp^{-1}(\tau)$ that lies in $T$, and 
choose components $\tau_1, \ldots, \tau_{2n-1}$, from 
among the remaining $2m-1$ components that lie in $O_m$. 
See Figure \ref{fig:Z^d_{m,n}}.

We will cut along each $\tau_i$ and reglue, but the
choice of gluing must be made with some care to ensure that 
the resulting 1-form has no more than two zeros.
To describe the gluing, we will suppose that for $i < 2m$, 
the $\tau_i$ are labeled in the order that they appear as one
winds around the zero of degree $2m-2$ on $O_m$. More precisely,
suppose that $\gamma: \Rbb/\Zbb \to T$ parameterizes the 
circle of radius $\epsilon/2$ centered at the origin $0 \in T$.
Then $\gamma$ lifts via $p$ to a map $\tgamma: \Rbb/(2m-1)\Zbb \to O_m$, 
and for each $i< 2n$, there exists a unique $t_i \in \Rbb/(2m-1)\Zbb$
such that $\tgamma(t_i) \in \tau_i$. By relabeling if necessary, we may 
assume that $i < j$ if and only if $t_i < t_j$.

Cut along each $\tau_i$ and reglue the bottom
of $\tau_i$ to the top of $\tau_{i+1}$. 
Let $q: Z^d_{m,n} \to \Cbb/(\Zbb+ i \Zbb)$ denote the 
resulting (connected) torus covering of degree $d$.
Because $2n$ is even, the point $q^{-1}(0)$ is a zero of order $2m-1$,
and if $Q$ denotes the other endpoint of $\tau$, then $q^{-1}(Q)$ 
is a single zero of order $2n-1$. The genus of $Z^d_{m,n}$ is $m+n$.

\begin{figure}
    \centering
    \includegraphics[scale=1.2]{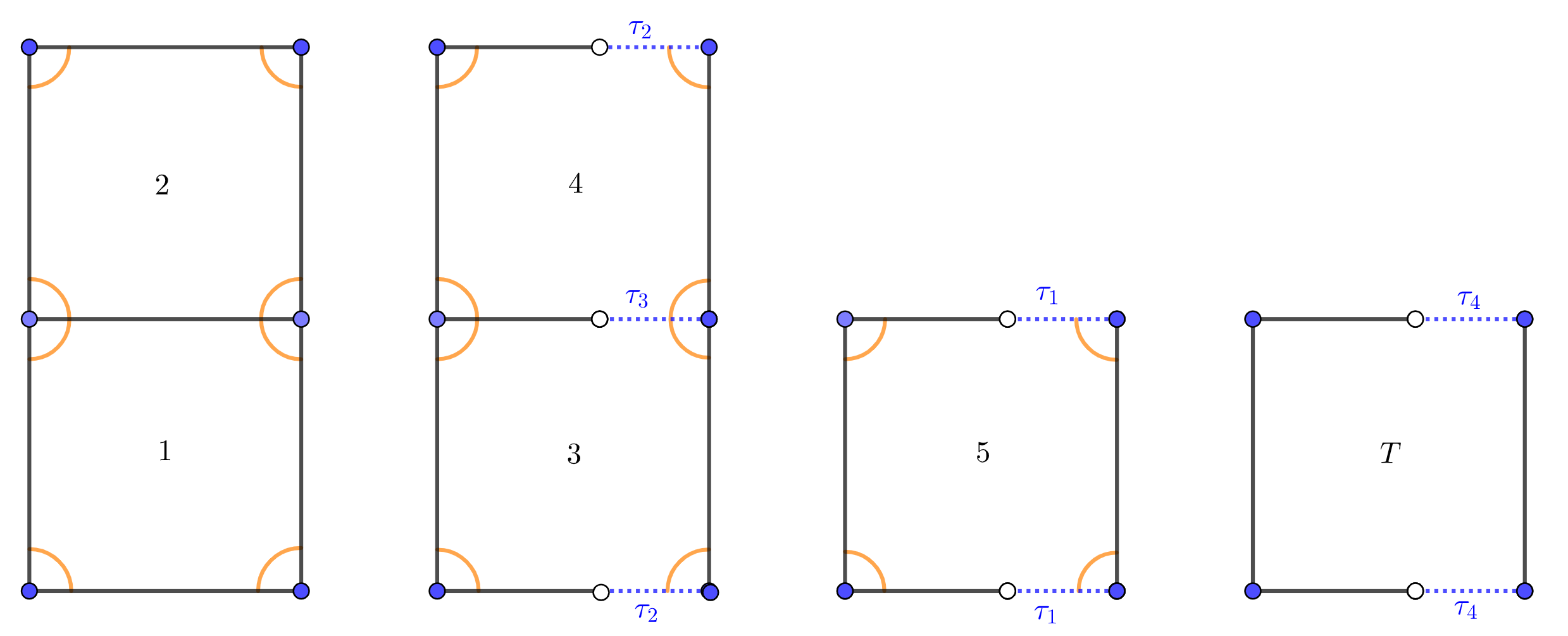}    
    \caption{The degree 6 torus cover $Z_{3,2} \in \Omega \Mcal(5,3)$ obtained
    from the surface $O_3 \in \Omega \Mcal(4)^{{\rm odd}}$. The purple 
    points correspond to the zero of order $5$, and the white 
    points correspond to the zero of order $3$. The labeling of the 
    $\tau_i$ is induced by the simple closed curve $\tgamma$ 
    on $O_3$ that winds clockwise around the pre-image of $0$.}
    \label{fig:Z^d_{m,n}}
\end{figure}


\subsection{An algorithm and examples}
\label{sec:algorithm}

In this section we describe
an algorithm for constructing a primitive torus cover of degree
$d$ in any desired connected component $K$ of a stratum 
$\Omega \mathcal{M}_g(\alpha)$. 
We then illustrate the algorithm 
with some examples.

Otherwise, we suppose that the desired divisor data 
$\alpha = (\alpha_1, ..., \alpha_n)$ satisfies 
\[
\alpha_1 \leq \alpha_2 \leq \cdots \leq \alpha_n,
\]
and if each $\alpha_i$ is even, 
we define $ j \in \Zbb/2 \Zbb$ by 
\[
    \theta~ 
    :=~  
    {\rm spin}~ +~ \frac{\alpha_1}{2}~ 
            +~ \frac{\alpha_2}{2}~ +~ \cdots~ +~ \frac{\alpha_{n-1}}{2} \mod 2
\]
where `spin' denotes the desired spin parity.

\begin{itemize}
\item If each $\alpha_i$ is an even integer, and

  \begin{itemize}

  \item $n=1$, then apply one of the constructions in  \S \ref{subsec:minimal-strata}.
  
  \item $n=2$ and $\alpha_1=\alpha_2$, then apply one of the constructions in
            \S \ref{subsec:two-zeros}.
            
  \item otherwise 
  
    \begin{itemize}

    \item if $\theta = 0$ mod $2$, then apply the surgery of \S    
        \ref{subsec:single-even-order} to the torus cover $E_g^d$ to add
        zeros of order $\alpha_1, \alpha_2, \ldots, \alpha_{n-1}$.
    
    \item if $\theta = 1$ mod $2$, then apply the surgery of \S    
        \ref{subsec:single-even-order} to the torus cover $O_g^d$ to add
        zeros of order $\alpha_1, \alpha_2, \ldots, \alpha_{n-1}$.
    
    \end{itemize}  
    
    \end{itemize}
    
\item otherwise (when some $\alpha_i$ is  odd) 

  \begin{itemize}
  
  \item if $\alpha_n$ is even, then apply the surgeries of \S    
        \ref{subsec:single-even-order} and \S \ref{subsubsec:pair_same_odd}
        to the torus cover $O_g^d$ to add
        zeros of order $\alpha_1, \alpha_2, \ldots, \alpha_{n-1}$.

  \item if $\alpha_n$ is odd, then some other zero, say $\alpha_j,$ is odd. 
        Begin with the  torus cover $Z^d_{m,n}$ where $m=(\alpha_n+1)/2$ and $n=(\alpha_j+1)/2$, and apply the surgeries of \S   
        \ref{subsec:single-even-order} and \S \ref{subsubsec:pair_same_odd}
        to add zeros of order $\alpha_i$ for $i \neq j$ or $n$.
  
  \end{itemize}

\end{itemize}

\subsubsection{Torus covers in $\Omega\mathcal{M}_4(1,2,3)$}
Suppose that we wish to construct a degree $4$ torus cover $p$
so that $p^*(dz)$ has a zero of order 3, a zero of order 2,
and a zero of order 1. That is, we have $\alpha_3=3$ 
which is odd, and $\alpha_1=1$ is odd as well. 
Hence we begin with the torus cover $Z^4_{2,1}$ 
whose construction is described in \ref{subsec:highest-order-odd}.
Then we perform the surgery described in \ref{subsec:single-even-order}
to add a zero of order two. See Figure \ref{fig:1-2-3}. 
In more detail, the surface $Z^4_{2,1}$
is obtained by slitting the $L$-shaped surface $O_2$ that lies in 
$\Omega \Mcal_2 (2)$ along the segment $\tau_2$, slitting a square 
torus along a segment $\tau_2$, and then gluing the top (resp. bottom) of $\tau_1$
to the bottom (resp. top) of $\tau_2$. The resulting surface 
is cut along the segments $\sigma_0$ and $\sigma_1$ and the 
the left (resp. right) of $\sigma_0$
to the right (resp. right) of $\sigma_1$.

Note that by adjoining $d-4$ squares to the right---that is, by 
replacing the unit square torus with the rectangular 
torus $\Cbb/((d-3)\Zbb + i \Zbb)$---one
obtains a degree $d > 4$ torus covering in $\Omega\mathcal{M}_4(1,2,3)$.

\begin{figure}
    \centering
    \includegraphics[scale=1.25]{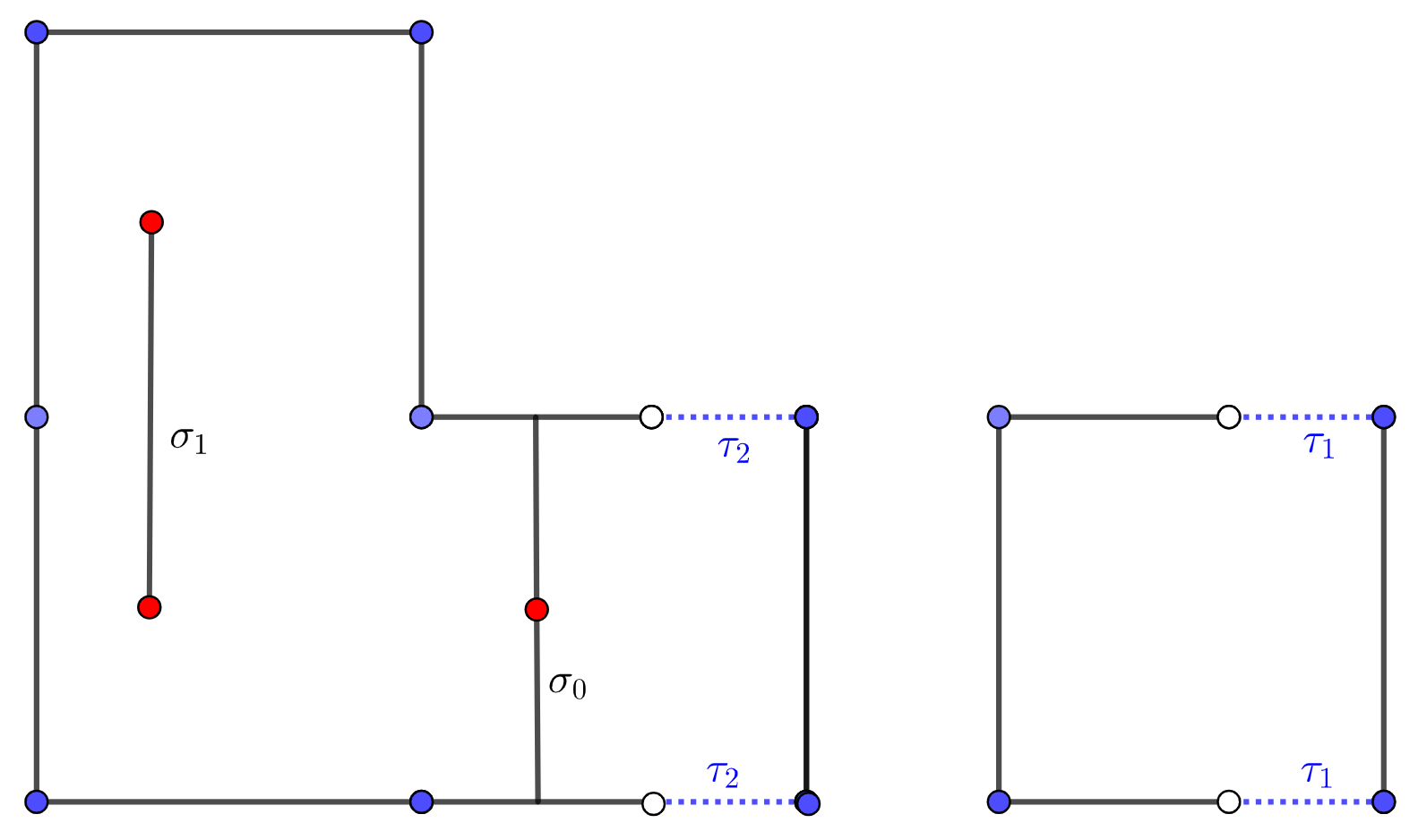}    
    \caption{A degree 4 torus cover in the stratum $\Omega \Mcal_4(1,2,3)$.
    The blue points correspond to a zero of order 3, the white
    points correspond to a zero of order 2, and the red points 
    correspond to a zero of order 1.}
    \label{fig:1-2-3}
\end{figure}


\subsubsection{Torus covers in $\Omega\mathcal{M}^{\mathrm{odd}}_7(2,4,6)$}

We describe the construction of a torus cover of degree 7 
that has odd spin parity and has one zero of order 6, one zero
of order 4, and one zero of order 2. Since $\theta = 1 + 2+1 = 0$ mod $2$,
we begin with the degree 7 torus cover $p:E_4^7 \to T$ 
such that $p^*(dz) \in \Omega \mathcal{M}_4(6)$, and we 
then we use the surgery of \S \ref{subsec:single-even-order}
to add zeros of order $2$ and $4$ while preserving the degree. 
See Figure \ref{fig:2-4-6}. In detail, `opposing sides' of the 
polygon in  Figure \ref{fig:2-4-6} are identified with the 
exception of the sides labeled $\gamma_i$ in which case we  
$\gamma_1$ identify with $\gamma_3$ and 
$\gamma_2$ identify with $\gamma_4$.
To add the additional zero of order $4$ (resp. $2$)
we cut along $\sigma_0$ and $\sigma_1$ (resp. $\sigma_0'$ and $\sigma_1'$)
and reglue. 

\begin{figure}
    \centering
    \includegraphics[scale=1.1]{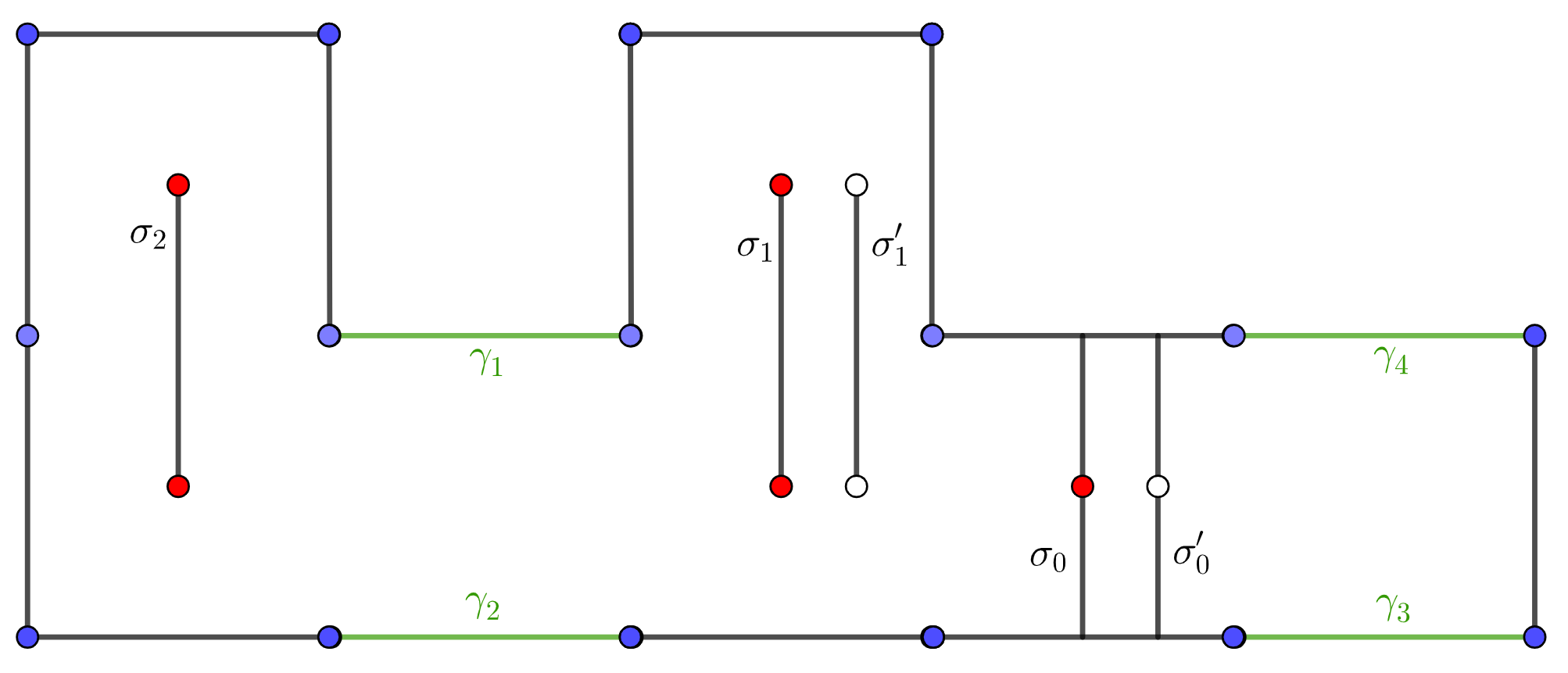}    
    \caption{A degree 7 torus cover in the odd component of 
    $\Omega\mathcal{M}_7(2,4,6)$.  The blue points correspond to
    the zero of order $6$, the red points correspond to the zero of 
    order $4$, the white points correspond to the zero of order $2$.}
    \label{fig:2-4-6}
\end{figure}

By adjoining $d-7$ additional squares to the right, one
obtains a degree $d > 7$ torus covering in $\Omega\mathcal{M}_4(1,2,3)$.


\end{document}